\documentclass[12pt,draftcls,onecolumn]{IEEEtran}

\IEEEoverridecommandlockouts                              

\usepackage{ifpdf}
\ifpdf
    \usepackage{graphicx}
    \usepackage{epstopdf}
    \def\Ebox#1#2{%
	\includegraphics[width=#1\hsize]{figures/#2}
    }
\else
    \usepackage{graphicx}
    \def\Ebox#1#2{%
	\includegraphics[width=#1\hsize]{figures/#2.eps}
    }
\fi

\usepackage{mathptmx} 
\usepackage{times} 
\usepackage{amsmath} 
\usepackage{amssymb}  
\usepackage{tabularx}
\usepackage{algorithm}
\usepackage{algorithmic}
\usepackage[dvipsnames,usenames]{color}
\usepackage[colorlinks,%
linkcolor=BrickRed,%
filecolor=BrickRed,%
citecolor=RoyalPurple,%
]{hyperref}

\usepackage{color}

\newcommand{\spm}[1]{{\bf \color{Mulberry}#1}}

\usepackage{tabularx}

\title{\LARGE \bf
Feedback Particle Filter}

\author{Tao Yang, Prashant G. Mehta and Sean P. Meyn
\thanks{T. Yang and P. G. Mehta are with the Coordinated Science Laboratory and
the Department of Mechanical Science and Engineering at the University of
Illinois at Urbana-Champaign (UIUC)
{\tt\small taoyang1@illinois.edu; mehtapg@illinois.edu}}
\thanks{S. P. Meyn is with the Department of Electrical and Computer
Engineering at the University of Florida
{\tt\small meyn@ece.ufl.edu}}%
\thanks{Financial support from the AFOSR
grant FA9550-09-1-0190  and the NSF grant EECS-0925534 is gratefully acknowledged.}%
\thanks{The conference version of this paper appeared
in~\cite{YangMehtaMeyn_acc11,YangMehtaMeyn_cdc11}.}%
}

\def\qed{\hspace*{\fill}~\IEEEQED\par\endtrivlist\unskip}

\def\sTriangle{\hbox{\small $\triangle$}}
\def\Lap{\Delta}

\def\hap{\hat{p}}

\def\ddx{\frac{\ud}{\ud x}}
\def\ddt{\frac{\ud}{\ud t}}

\def\Re{\mathbb{R}}

\def\KL{\,\text{\rm KL}}
\def\argmin{\mathop{\text{\rm arg\,min}}}
\def\ind{\text{\rm\large 1}}
\def\varble{\,\cdot\,}

\def\pyx{p_{\text{\tiny$ Y | X$}}}
\def\py{p_{\text{\tiny$Y$}}}

\def\ddx{\frac{\ud}{\ud x}}
\def\ddt{\frac{\ud}{\ud t}}

\usepackage{eufrak}

\def\varble{\,\cdot\,}

\def\Inov{I} 

\def\Lemma#1{Lemma~\ref{#1}}
\def\Proposition#1{Prop.~\ref{#1}}
\def\Theorem#1{Thm.~\ref{#1}}

\def\Sec#1{Sec.~\ref{#1}}
\def\Fig#1{Fig.~\ref{#1}}
\def\Appendix#1{App.~\ref{#1}}

\def\notes#1{\marginpar{\tiny #1}\typeout{Notes!
Notes!
Notes!
}}
\renewcommand{\notes}[1]{\typeout{notes!}}

\def\FRAC#1#2#3{\genfrac{}{}{}{#1}{#2}{#3}}

\def\half{{\mathchoice{\FRAC{1}{1}{2}}%
{\FRAC{2}{1}{2}}%
{\FRAC{3}{1}{2}}%
{\FRAC{4}{1}{2}}}}

\def\bbbone{{\mathchoice {\rm 1\mskip-4mu l} {\rm 1\mskip-4mu l}
{\rm 1\mskip-4.5mu l} {\rm 1\mskip-5mu l}}}
\def\ind{\bbbone}

\newcommand{\field}[1]{\mathbb{#1}}
\newcommand{\tr}{\mbox{tr}}

\def\Re{\field{R}}
\def\v{{\sf K}}
\def\w{{\sf \Omega}}

\def\Sec#1{Sec.~\ref{#1}}

\def\eqdef{\mathrel{:=}}

\def\transpose{{\hbox{\rm\tiny T}}}

\def\clB{{\cal B}}
\def\clE{{\cal E}}
\def\clL{{\cal L}}
\def\clP{{\cal P}}
\def\clZ{{\cal Z}}
\def\clY{{\cal Y}}

\def\bfmath#1{{\mathchoice{\mbox{\boldmath$#1$}}%
{\mbox{\boldmath$#1$}}%
{\mbox{\boldmath$\scriptstyle#1$}}%
{\mbox{\boldmath$\scriptscriptstyle#1$}}}}

\def\Inov{I} 

\def\bfmX{\bfmath{X}}

\def\bfmZ{\bfmath{Z}}

\newcounter{rmnum}
\newenvironment{enum}{\begin{list}{{\upshape \arabic{rmnum})}}{\usecounter{rmnum}
\setlength{\leftmargin}{6pt}
\setlength{\rightmargin}{4pt}
\setlength{\itemindent}{-1pt}
}}{\end{list}}

\newenvironment{romannum}{\begin{list}{{\upshape (\roman{rmnum})}}{\usecounter{rmnum}
\setlength{\leftmargin}{6pt}
\setlength{\rightmargin}{4pt}
\setlength{\itemindent}{-1pt}
}}{\end{list}}

\newcounter{anum}

\renewcommand{\baselinestretch}{1.25}
\renewcommand{\varepsilon}{\text{\usefont{OML}{cmr}{m}{n}\symbol{15}}}

\newcommand{\ud}{\,\mathrm{d}}

\def\Expect{{\sf E}}

\def\Prob{{\sf P}}

\def\Expect{{\sf E}}

\def\spm#1{\notes{\textcolor{blue}{SPM: #1}}}

\def\wham#1{\smallbreak\noindent\textbf{\textit{#1}}}

\usepackage{ulem}

\newtheorem{theorem}{Theorem}[section]

\newtheorem{proposition}[theorem]{Proposition}
\newtheorem{lemma}[theorem]{Lemma}

\newtheorem{definition}{Definition}

\newtheorem{remark}{Remark}

\begin{document}
\normalem
\maketitle

\vspace{-0.5in}
\begin{abstract}

A new formulation of the particle filter for nonlinear
filtering is presented, based on concepts from optimal control,
and from the mean-field game theory. The optimal control is chosen so
that the posterior distribution of a particle matches as closely as possible
the posterior distribution of the true state given the observations.  This is achieved by introducing
a cost function, defined by  the Kullback-Leibler (K-L) divergence between the
actual posterior, and the posterior of any particle.

The optimal control input is characterized by a certain Euler-Lagrange (E-L) equation, and is shown to admit an innovation error-based feedback structure.  For diffusions with continuous observations, the value of the optimal control solution is ideal.  The two posteriors match exactly, provided they are initialized with identical priors.  
 The {\em feedback particle filter} is defined by a family of stochastic systems, each evolving under this optimal control law.
 \spm{defining FPF as a control system seemed odd to me, so I changed the sentence.}

%

A numerical algorithm is introduced and implemented in two general examples,  and a neuroscience application involving coupled oscillators. Some preliminary numerical comparisons between the feedback particle filter and the bootstrap particle filter are described.

\end{abstract}

\section{Introduction} 
\label{sec:intro}

\def\fp{(\ref{eqn:Signal_Process}, \ref{eqn:Obs_Process})}

We consider a scalar filtering problem:
\begin{subequations}
\begin{align}
\ud X_t &= a(X_t)\ud t + \sigma_B \ud B_t,
\label{eqn:Signal_Process}
\\
\ud Z_t &= h(X_t)\ud t + \sigma_W \ud W_t,
\label{eqn:Obs_Process}
\end{align}
\end{subequations}
where $X_t\in\Re$ is the state at time $t$, $Z_t \in\Re$ is the
observation process, $a(\varble)$, $h(\varble)$ are $C^1$
functions, and $\{B_t\}$, $\{W_t\}$ are mutually independent
standard Wiener processes.  Unless otherwise noted, the stochastic
differential equations (SDEs) are expressed in It\^{o} form.

The objective of the filtering problem is to 
compute or approximate 
the
posterior distribution of $X_t$ given the history $\clZ_t :=
\sigma(Z_s:  s \le t)$. 
The posterior $p^*$ is defined so
that, for any measurable set $A\subset \Re$,
\begin{equation}
\int_{x\in A} p^*(x,t)\, \ud x   = \Prob\{ X_t \in A\mid \clZ_t \}.
\label{e:pXDef}
\end{equation}
The filter is infinite-dimensional since it
defines the evolution, in  the space of probability measures,
of $\{p^*(\varble ,t) : t\ge 0\}$.  If $a(\varble)$, $h(\varble)$ are linear functions, the solution is given by the
finite-dimensional Kalman filter. The theory of nonlinear
filtering is described in the classic monograph~\cite{kal80}.

The article~\cite{budchelee07} surveys numerical methods to approximate the nonlinear filter.
One approach described in this survey is particle filtering. 

The particle filter is a simulation-based algorithm to approximate the
filtering task~\cite{HandschinMayne69,gorsalsmi93,DouFreGor01}. The key step is the
construction of $N$ stochastic processes $\{X^i_t : 1\le i \le N\}$.
The value $X^i_t \in \Re$ is the state for the $i^{\text{th}}$
particle at time $t$. For each time $t$, the empirical distribution
formed by, the ``particle population'' is used to approximate the conditional distribution.  Recall that this is  defined for any measurable set $A\subset\Re$ by,
\begin{equation}
p^{(N)}(A,t) = \frac{1}{N}\sum_{i=1}^N \ind\{ X^i_t\in A\}.
\label{e:piN}
\end{equation}

\smallskip

A common approach in particle filtering is called {\em
sequential importance sampling}, where particles are generated
according to their importance weight at every time stage~\cite{DouFreGor01,budchelee07}. By
choosing the sampling mechanism properly, particle filtering
can approximately propagate the posterior distribution, with the
accuracy improving as $N$ increases~\cite{Crisan_Doucet_02}.  

\smallskip

The objective of this paper is to introduce an alternative approach to
the construction of a particle filter
for~\eqref{eqn:Signal_Process}-\eqref{eqn:Obs_Process} inspired by mean-field optimal
control techniques;
cf.,~\cite{huacaimal07,YinMehtaMeynShanbhag12TAC}.
In this approach,
the model for the $i^{\text{th}}$ particle is defined by a controlled system,
\begin{equation}
\ud X^i_t = a(X^i_t) \ud t + \sigma_B \ud B^i_t + \ud U^i_t ,
\label{eqn:particle_filter}
\end{equation}
where $X^i_t \in \Re$ is the state for the $i^{\text{th}}$
particle at time $t$, $ U^i_t$ is its
control input, and $\{B^i_t\}$ are mutually independent
standard Wiener processes.  Certain additional assumptions
are made regarding admissible forms of control input.
\spm{Please do not delete this note:  What happens if we have
correlation?  There is NO reason to assume that the $B^i$ are independent.}

Throughout the paper we denote conditional distribution of a particle
$X^i_t$ given $\clZ_t$ by $p$ where, just as in the definition of $p^*$:
\begin{equation}
\int_{x\in A} p(x,t)\, \ud x   = \Prob\{ X^i_t \in A\mid \clZ_t \}.
\label{e:pDef}
\end{equation}
The initial conditions $\{X^i_0\}_{i=1}^N$  are assumed to be i.i.d., and drawn from
initial distribution $p^*(x,0)$ of $X_0$ (i.e., $p(x,0)=p^*(x,0)$).
 \spm{Again, i.i.d.\ is NOT necessary!!  Correlation may reduce variance  -- we don't know.}




The control problem is to choose the control input $U^i_t$ so
that $p$ approximates $p^*$, and consequently $p^{(N)}$
(defined in \eqref{e:piN}) approximates $p^*$ for large $N$. 
The synthesis of the control input  
is cast
as an optimal control
problem, with the Kullback-Leibler metric serving as the cost
function. The optimal control input is obtained via analysis of
the first variation.

The main result of this paper is to derive an explicit formula for the
optimal control input, and demonstrate that under general conditions
we obtain an exact match:  $p=p^*$ under optimal control.    The
optimally controlled dynamics of the $i^{\text{th}}$ particle have the
following It\^{o} form,
\spm{careful you aren't using ${}'$ for transpose anywhere!  Use ${}^\transpose$ instead.}
\begin{equation}
\ud X^i_t = a ( X^i_t) \ud t + \sigma_B \ud B^i_t +
\underbrace{\v(X^i_t,t) \ud \Inov^i_t + \w(X^i_t,t) \ud t}_{\text{optimal control, $ \ud U^{i*}_t $} },
\label{eqn:particle_filter_nonlin_intro}
\end{equation}
in which $\w(x,t):=\frac{1}{2} \sigma_W^2
  \v(x,t) \v'(x,t)$, $\v'(x,t) = \frac{\partial \v}{\partial x}(x,t)$,
and  $\Inov^i$ is similar to the \textit{innovation process} that appears in the nonlinear filter,\begin{equation}
\ud \Inov^i_t \eqdef \ud Z_t - \frac{1}{2}    (h(X^i_t) + \hat{h}_t) \ud t,
\label{e:in}
\end{equation}
where   $\hat{h}_t := {\sf E} [h(X^i_t)|\clZ_t] = \int h(x) p(x,t) \ud x$.  
In a numerical
implementation, we approximate
\spm{too important to hide!}
\begin{equation}
\hat{h}_t 
	\approx  \hat{h}^{(N)}_t 
 	:= 
	\frac{1}{N}
\sum_{i=1}^N h(X^i_t) \, .
\label{e:hah}
\end{equation}

The gain function $\v$ is shown to be the solution to
the following Euler-Lagrange boundary value problem (E-L BVP):
\begin{equation}
-\frac{\partial}{\partial x}\left( \frac{1}{p(x,t)} \frac{\partial
  }{\partial x} \{ p(x,t) \v(x,t) \}\right) = \frac{1}{\sigma_W^2} h'(x),
\label{eqn:EL_v_intro}
\end{equation}
with boundary conditions $\lim_{x\rightarrow \pm \infty}
p(x,t)\v(x,t)= 0$, where $h'(x)=\ddx h\, (x)$.

Note that the gain function needs to be obtained for each value of
time $t$.  If the right hand side of \eqref{eqn:EL_v_intro} is non-negative valued, it then  
follows from the minimum
principle for elliptic BVPs that the gain function $\v$ is non-negative valued \cite{Evans:98}.  
\spm{Asking again in July of 2012:  is it strictly positive when $h$ is increasing and non-constant? 
If so, we should say that $\v$ takes on positive values.}


The contributions of this paper are as follows:

\noindent {$\bullet$ \bf Variational Problem.}  The construction of
the feedback particle filter is based on a variational problem,
where the cost function is the Kullback-Leibler
(K-L) divergence between $p^*(x,t)$ and $p(x,t)$.  The feedback
particle
filter~\eqref{eqn:particle_filter_nonlin_intro}-\eqref{eqn:EL_v_intro},
including the formula~\eqref{e:in} for the innovation error and the E-L~BVP~\eqref{eqn:EL_v_intro}, 
is obtained via analysis of the
first variation.
\spm{Omitted "We" in many places}

\noindent {$\bullet$ \bf Consistency.}   
The particle filter
model~\eqref{eqn:particle_filter_nonlin_intro}
is consistent with nonlinear filter in the following sense: Suppose the gain function $\v(x,t)$ is obtained as the solution to~\eqref{eqn:EL_v_intro},  and the priors are consistent, $p(x,0)=p^*(x,0)$. 
Then, for all $t\ge 0$ and all $x$,
\[
p(x,t) = p^*(x,t).
\] 

\noindent {$\bullet$ \bf Algorithms.}
Numerical techniques are proposed  for synthesis of the gain function
  $\v(x,t)$.  If $a(\cdot)$ and $h(\cdot)$ are linear and the density
  $p^*$ is Gaussian, then the gain function is simply the Kalman gain.
  At time $t$, it is a constant given in terms of variance alone.
  The variance is approximated empirically as a sample covariance.

In the nonlinear case,  numerical approximation techniques are described.  Other approaches using sum of Gaussian approximation also exist but are omitted on account of space. Details for the latter can be found in~\cite{YangMehtaMeyn_cdc11}.

\smallskip

In recent decades, there have been many important advances in
importance sampling based approaches for particle
filtering; cf.,~\cite{DouFreGor01,budchelee07,Xiong}.
A crucial distinction here is that there is no
resampling of particles.

We believe that the introduction of control in the feedback particle
filter has several useful features/advantages:

\spm{too timid: In certain cases, there may be
some advantages to using a feedback particle filter
formulation:}
\spm{eliminated bullets to separate from contributions above}

\wham{Does not require sampling.}  There is no re-sampling required as in
the conventional particle filter.  This property allows the feedback particle filter to be
flexible with regards to implementation and does not suffer
from sampling-related issues. 

\wham{Innovation error.}
The innovation error-based feedback structure is a key feature of the
feedback particle filter \eqref{eqn:particle_filter_nonlin_intro}.
The innovation error in  \eqref{eqn:particle_filter_nonlin_intro}
is
based on the average value
of the prediction $h(X^i_t)$ of the $i^{\text{th}}$-particle
and the prediction $\hat{h}_t$ due to the entire population.

The feedback structure is easier to see when the filter is expressed
in its Stratonovich form:
  \spm{Talk about Kick Ass!}
\begin{equation}
\ud X^i_t = a(X^i_t) \ud t + \ud B^i_t +  \v(X^i,t) \circ \left( \ud Z_t - \frac{1}{2}    (h(X^i_t) + \hat{h}_t) \ud t\right).
\label{eqn:FPF_Strato}
\end{equation}
Given that the Stratonovich form provides a mathematical
interpretation of the (formal) ODE model \cite[Section 3.3]{Oksendal_book}, we also obtain the ODE model of the filter.
Denoting $Y_t \doteq \frac{\ud Z_t}{\ud t} $ and white noise
process $\dot{B}^i_t \doteq \frac{\ud B_t^i}{\ud t} $, the ODE model of the
filter is given by,
\begin{equation}
\frac{\ud X^i_t}{\ud t} = a(X^i_t) + \dot{B}^i_t +  \v(X^i,t) \cdot \left( Y_t - \frac{1}{2}    (h(X^i_t) + \hat{h}_t)\right). \nonumber
\end{equation}

The feedback particle filter thus provides for a generalization of the
Kalman filter to nonlinear systems, where the innovation error-based
feedback structure of the control is preserved (see
\Fig{fig:fig_FPF_KF}).  For the linear case, the optimal gain function
is the Kalman gain.  For the nonlinear case,
the Kalman gain is replaced by a nonlinear function of the state.

\wham{Feedback structure.} 
Feedback is important on account of the issue of {\em robustness}.
A filter is based on an idealized model of the
underlying dynamic process that is often nonlinear, uncertain and
time-varying.  The self-correcting property of the feedback provides  
robustness, allowing one to tolerate a degree of uncertainty inherent
in any model. 

In contrast, a conventional particle filter is based upon 
importance sampling.  Although the innovation error is central to the Kushner-Stratonovich's stochastic partial
differential equation (SPDE) of nonlinear filtering, it is conspicuous
by its absence in a conventional particle filter.

Arguably, the structural aspects of the Kalman filter have been as important as the algorithm
itself in design, integration, testing and operation of the overall
system.  Without such structural features, it is a challenge to
create scalable cost-effective solutions.  

The ``innovation'' of the feedback particle filter lies in the
(modified) 
definition of innovation error for a particle filter.  Moreover, the feedback control structure that existed thusfar only for Kalman filter
now also exists for particle filters (compare parts~(a) and~(b) of Fig.~\ref{fig:fig_FPF_KF}).         


\wham{Variance reduction.} 
Feedback can
help reduce the high variance that is sometimes observed in the
conventional particle filter. Numerical results in \Sec{sec:numerics}
support this claim ---  See \Fig{fig:fig_comp}
for a comparison of the feedback particle filter and the
bootstrap filter. 
\spm{However,
it is likely that the stability of the FPF will depend on the
stability properties of the nonlinear filter.   This is an open
question, and the possibility of improving variance at the cost
of bias is another big open question.  Let's discuss.
\\
Back to this in July...  The instability of the BPF is hidden.  When it comes up later it is very confusing.  I'll attempt to clarify, but please have a look.}

\wham{Ease of design, testing and operation.}  On account of structural
features, feedback particle filter-based solutions are expected to be more robust,
cost-effective, and easier to debug and implement.

\begin{figure}
    \centering
\vspace{-0.25in}
    \includegraphics[width=\columnwidth]{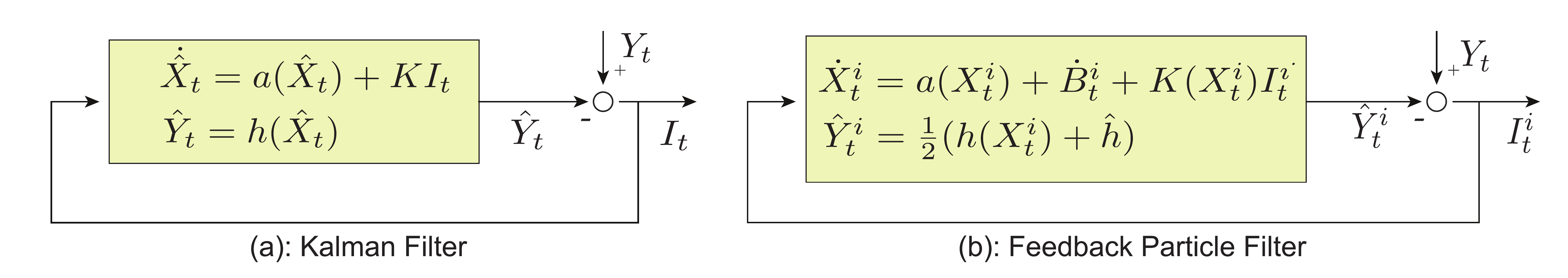}
    \vspace{-0.25in}
\caption{Innovations error-based feedback structure for (a) Kalman
  filter and (b) nonlinear feedback particle filter.
}\vspace{-.4cm}
    \label{fig:fig_FPF_KF}
\end{figure}

\wham{Applications.}
Bayesian inference is an important paradigm used to model
functions  
of certain neural circuits in the brain~\cite{Bayesian_Brain}.
Compared to techniques that rely on importance sampling, a
feedback particle filter may provide a more neuro-biologically
plausible model to implement filtering and inference
functions~\cite{YangMehtaMeyn_acc11}.  This is illustrated here
with the aid of a filtering problem involving nonlinear
oscillators.  Another application appears in~\cite{TiltonLizMehta12ACC}.



%

\subsection{Comparison with Relevant Literature}

Our work is motivated by recent development in mean-field
games, but the focus there has been primarily on optimal control~\cite{huacaimal07,YinMehtaMeynShanbhag12TAC}.

In nonlinear filtering, there are two  directly related works:
Crisan and Xiong~\cite{CriXiong09}, and   Mitter and
Newton~\cite{MitterNewton04}. 
In each of these papers,   a controlled system is introduced, of the form
\begin{equation*}
\ud X^i_t = \left(a(X^i_t) + u(X^i_t,t)\right) \ud t + \sigma_B \ud B^i_t.
\end{equation*}
The objective is to choose the control input to obtain a solution of
the nonlinear filtering problem.

The approach in~\cite{MitterNewton04} is based on consideration of a
finite-horizon optimal control problem.  It leads to an HJB equation
whose solution yields the optimal control input.

The work of Crisan and Xiong is closer to our
paper in terms of both goals and approaches.
\spm{note subtle change here on "goals"}
Although we were not aware of their work prior to submission
of our original conference papers~\cite{YangMehtaMeyn_acc11,YangMehtaMeyn_cdc11},
Crisan and Xiong provide an explicit expression for a control law
that is similar to the feedback particle filter, with some important differences.  One,
the considerations of Crisan and Xiong (and also of Newton and
Mitter) require introduction of a {\em smooth approximation} of
the process ``$\frac{\ud Z}{\ud t} - \hat{h}_t$,'' which we
avoid with our formulation.  Two, the filter derived in Crisan
and Xiong has a structure based on a gain feedback with respect
to the smooth approximation, while the feedback particle filter is
based on the formula for innovation error $I_t^i$ as given
in~\eqref{e:in}.  This formula is fundamental to
construction of particle filters in continuous time settings.
We clarify here that the formula for innovation error is {\em not}
assumed,  comes about as a result of the analysis of the variational problem.
\spm{Why don't you like "the"?
\\
We have "the KF", so why not "the FPF"?}

Remarkably, both the feedback particle filter and
Crisan and Xiong's filter require solution of the same
boundary value problem, and as such have the same computational
complexity.  The BVP is solved to obtain the gain function.  However,
the particular solution described in Crisan and Xiong for the BVP may
not work in all cases, including the linear Gaussian case.
Additional discussion appears in~\Sec{sec:comparison}.
\spm{I feel this is far too weak.  We should say that their proposed representation of the gain is not correct in many cases}

Apart from these two works, Daum and Huang have introduced the
{\em information flow filter} for the continuous-discrete time filtering
problem~\cite{DaumHuang10}.  Although an explicit formula for the
filter is difficult to obtain, a closely related form of the boundary value
problem appears in their work.  There is also an important
discussion of both the limitations of the conventional particle
filter, and the need to incorporate feedback to ameliorate
these issues. Several numerical experiments are presented that
describe high variance and robustness issues, especially where
signal models are unstable.
These results provide significant motivation to the work described here.

\subsection{Outline}

The variational setup is described in~\Sec{sec:prelim}:
It begins with a
discussion of the continuous-discrete filtering problem: the
equation for dynamics is defined by~\eqref{eqn:Signal_Process}, but the
observations are made only at discrete times.  The
continuous-time filtering problem
(for~\eqref{eqn:Signal_Process}-\eqref{eqn:Obs_Process}) is obtained as a limiting case of
the continuous-discrete problem.

The feedback particle filter is introduced in
\Sec{sec:continuous_continuous}.
Extension to the multivariable case is briefly described in~\Sec{sec:FPF_multi-variable}, 
followed by a comparison with Crisan and Xiong's filter in~\Sec{sec:comparison}.
\spm{I'm not sure this is useful here:
\\
Additional details are planned for a future publication.
}

Algorithms are discussed in
\Sec{sec:nonlinear}, and numerical examples are described in
\Sec{sec:numerics}, 
 including the neuroscience application
involving coupled oscillator models.   These models (also considered in
our earlier mean-field control paper~\cite{YinMehtaMeynShanbhag12TAC}) provided
some of the initial motivation for the present work. 

\spm{Not convincing to me:
Apart from
serving a pedagogical purpose, the scalar case is especially
relevant to filtering with coupled oscillator models. 
}



\section{Variational Problem}
\label{sec:prelim}

The control problem posed by any one of the $i^{\text{th}}$ particles
can be cast as a partially observed optimal control problem.  The observations are given by
$\{X^i_t , Z_t \}$, and the state process is two-dimensional,
$\{X^i_t , X_t \}$. In partially observed optimal control
problems, it is typical to take the ``belief state'' $p^*_t$ as
the state process, which is known to serve as a sufficient
statistic for optimal control under general conditions.   Since
our cost function is taken as the KL divergence between $p^*_t$
and $p_t$ (defined in \eqref{e:pXDef} and \eqref{e:pDef},
respectively),  a natural state process for the purposes of
optimal control is the triple $\{X^i_t , p_t, p^*_t \}$. 

The precise formulation of the optimal control problem begins with the continuous time model,  with sampled observations. The equation for dynamics is given
by~\eqref{eqn:Signal_Process}, and the observations are made
only at discrete times $\{t_n\}$:
\begin{align}
Y_{t_n} &= h(X_{t_n}) + W_{t_n}^{\sTriangle},\label{eqn:meas_2}
\end{align}
where $\sTriangle := t_{n+1} - t_n$ and $\{W_{t_n}^{\sTriangle}\}$ is i.i.d and drawn from
$\textbf{\emph{N}}(0,\frac{\sigma^2_W}{\sTriangle})$.
\spm{Conventions I'd prefer:  $\nabla^2$ is the matrix of second derivatives, and $\sTriangle$ the Laplacian, and $\sTriangle$ is used for  increments}

The particle model in this case is a hybrid dynamical system:
For $t\in[t_{n-1},t_{n})$, the $i^{\text{th}}$ particle evolves
according to the stochastic differential equation,
\begin{equation}
 \ud X^i_t = a(X^i_t) \ud t+\sigma_B \ud B^i_t \, ,\quad t_{n-1} \le t < t_{n}\, ,
\label{eqn:dyn_part_hyb}
\end{equation}
where the initial condition $X^i_{t_{n-1}}$ is given. At time
$t=t_{n}$ there is a potential jump that is determined by the
input  $U^i_{t_n}$:
\begin{equation}
X^i_{t_n} = X^i_{t_n^-} + U^i_{t_n}\, ,
\label{eqn:dyn_part_hyb_jump}
\end{equation}
where  $X^i_{t_n^-}$   denotes the right limit of $\{   X^i_t :
t_{n-1} \le t < t_{n}\}$.  The specification
\eqref{eqn:dyn_part_hyb_jump}  defines the initial condition
for the process on the next interval $[t_n,t_{n+1})$.

The filtering problem is to construct a control law that
defines  $\{U^i_{t_n}: n\ge 1\}$ such that $p(\varble,t_n)$ approximates $p^*(\varble,t_n)$
 for each $n\ge 1$.
\spm{This is too vague: the resulting empirical distribution $p^{(N)}$ approximates the
posterior distribution of $X_{t_n}$ given the history $\clZ_n:= \sigma (Z_{t_k} : k\le n)$.
}
To solve this problem we first
define ``belief maps'' that propagate the conditional
distributions of $\bfmX$ and $\bfmX^i$.

\subsection{Belief Maps}
\label{s:believe}

The observation history is denoted 
$\clY_n := \sigma\{Y_{t_i}:
i \leq n,i\in \mathbb{N}\}$. 
For each $n$, various conditional distributions are considered:
\begin{enum}
\item   $p_n^*$ and $p_n^{*-}$:
 The conditional distribution of $X_{t_n}$ given $\clY_n $ and $\clY_{n-1}$, respectively.  
 
\item $p_n$ and $p_n^{-}$:
The conditional distribution of
$X^i_{t_n}$ given $\clY_n $ and $\clY_{n-1}$, respectively.
\end{enum}

These densities evolve according to recursions of the form,
\begin{equation}
p^*_n=\clP^*(p^*_{n-1}, Y_{t_n}),\qquad
p_n=\clP(p_{n-1}, Y_{t_n})\, .
\label{e:pnMaps}
\end{equation}
The mappings $\clP^*$ and $\clP$ can be decomposed into two
parts.   The first part is identical for each of these
mappings:   the transformation that takes $p_{n-1}$ to $p_n^-$
coincides with the mapping from $p_{n-1}^*$ to $p_n^{*-}$. In
each case it is defined by  the Kolmogorov forward
equation associated with
the diffusion on $[t_{n-1},t_n)$.

The second part of the mapping is the transformation
that takes $p_n^{*-}$ to $p_n^*$, which is obtained from Bayes' rule:
Given the observation $Y_{t_n}$ made at time
$t=t_n$, 
\spm{too repetitive: the pdf for the actual state is updated using Bayes rule:}
\begin{equation}
  	p_n^*(s) = \frac{ p_n^{*-}(s)\cdot \pyx(Y_{t_n}|s)}{\py(Y_{t_n})}, \quad s\in\Re,
\label{eqn:pks_1}
\end{equation}
where $\py$ denotes the pdf for $Y_{t_n}$,  and $\pyx(\varble
\mid s)$ denotes the conditional distribution of $Y_{t_n}$
given $X_{t_n}=s$.  Applying \eqref{eqn:meas_2} gives,
\[
\pyx(Y_{t_n}\mid s) = \frac{1}{\sqrt{2\pi \sigma^2_W/\sTriangle}} \exp \left( -
\frac{(Y_{t_n}-h(s))^2}{2\sigma^2_W/\sTriangle} \right).
\]
Combining \eqref{eqn:pks_1} with the forward equation defines $\clP^*$.

The transformation that takes $p_n^{-}$ to $p_n$ depends upon
the  choice of control $U^i_{t_n}$
in~\eqref{eqn:dyn_part_hyb_jump}.  At time $t=t_n$, we seek a control
input $U^i_{t_n}$ that is   {\em admissible}.
\spm{don't like: The space of admissible functions is denoted as $C^2_b$.
\\
Note that "admissible function" is only used a few times, so let's remove "function" and replace by "input"
\\
HELP!   Please check my new terminology, and new definition below.}

\begin{definition}[Admissible Input]
The control sequence $\{U^i_{t_n} : n\ge 0\}$ is {\em admissible} if there is a sequence of maps $\{v_n(x;y_0^n)\}$ such that
$ U^i_{t_n} = v_n(X_{t_n^-}^i,Y_{t_0},\dots,Y_{t_n})$ for each $n$, and moreover,
\begin{romannum}
\item
	$ \Expect[|U^i_{t_n} |]<\infty$, and with probability one,
\[
\lim_{x\rightarrow\pm\infty} v_n(x, Y_{t_0},\dots,Y_{t_n}) p_n^{-} (x) = 0.
\]

\item
	$v_n$ is twice continuously differentiable as a function of $x$.

\item
	$1+v_n'(x)$ is non-zero for all $x$, where   $v_n' (x) = \ddx v_n(x)$.
\end{romannum}
\qed
\end{definition}

We will suppress the dependency of $v_n$ on the observations (and often the time-index $n$), writing
$U^i_{t_n} =v(x)$ when $X^i_{t_n^-}=x$.  Under the assumption that  $1+v'(x)$ is non-zero for
all $x$, we can write,
\begin{equation}
p_n(x^+) =   \frac{p_n^-(x)}{|1 + v' (x)|}\, ,\quad \hbox{where $x^+=x+v(x)$.}
\label{eqn:pks_2}
\end{equation}

\subsection{Variational Problem}

Our goal is to choose an admissible input so that the
mapping $\clP$ approximates the mapping $\clP^*$
in~\eqref{e:pnMaps}. More specifically, given the pdf $p_{n-1}$ we have
already defined the mapping $\clP$ so that $ p_n=\clP(p_{n-1},
Y_{t_n})$.   We denote $\hap^*_n = \clP^*(p_{n-1}, Y_{t_n})$,
and choose $v_n$ so that these pdfs are as close as
possible. We approach this goal through the formulation of an
optimization problem with respect to the KL
divergence metric.  That is, at time $t=t_n$, the function $v_n$ is the solution to the following optimization
problem,
\begin{equation}
v_n(x) = \argmin_v \KL \left( p_n \| \hap^*_n \right).
\label{eqn:optimiz_problem}
\end{equation}
Based on the definitions,  for any $v$ the KL divergence can be expressed,
\begin{equation}
\KL  ( p_n \| \hap^*_n  ) = -\int_{\Re} p_n^-(x) \Bigl\{ \ln |1+v' (x)| +\ln\left({p}_n^-(x+v(x)) \pyx(Y_{t_n}|x+v(x))\right) \Bigr\}\ud x + C,
\label{eqn:KLD}
\end{equation}
where $C = \int_{\Re} p_n^-(x)\ln(p_n^-(x)\py(Y_{t_n}))\ud x$
is a constant that does not depend on $v$; cf.,
\Appendix{cal_KL} for the calculation.

The   solution to \eqref{eqn:optimiz_problem} is described in the following proposition,
whose proof appears in \Appendix{der_EL}.
\begin{proposition}\label{prop one}
Suppose that the admissible input is obtained as the solution to the sequence of optimization problems~\eqref{eqn:optimiz_problem}.
Then for each $n$, the function $v=v_n$  is a solution of the following Euler-Lagrange (E-L) BVP:
\begin{equation}
\ddx \left(\frac{{p}_n^-(x)}{|1+v^{'}(x)|}\right)
=
{p}_n^-(x)\frac{\partial}{\partial v} \left( \ln({p}_n^-(x+v) \pyx(Y_{t_n}|x+v))\right),
\label{eqn:EL_eqn}
\end{equation}
with boundary condition $\lim_{x\rightarrow\pm\infty} v(x)
p_n^{-}(x) = 0$.
\qed
\end{proposition}

We refer to the minimizer as the {\em optimal control
function}.  Additional details on the continuous-discrete time
filter appear in our conference paper~\cite{YangMehtaMeyn_acc11}.

\section{Feedback Particle Filter}
\label{sec:continuous_continuous}

We now consider the continuous time filtering
problem~\fp~introduced in \Sec{sec:intro}.

\subsection{Belief State Dynamics and Control Architecture}
The model for the particle filter is given by the It\^{o} diffusion,
\begin{equation}
\ud X^i_t = a(X^i_t) \ud t + \sigma_B \ud B^i_t + \underbrace{u(X^i_t, t) \ud t + \v (X^i_t,t)
\ud Z_t}_{\ud U^i_t}, \label{eqn:particle_model}
\end{equation}
where $X^i_t \in \Re$ is the state for the $i^{\text{th}}$
particle at time $t$, and $\{B^i_t\}$ are mutually independent
standard Wiener processes.  We assume the initial conditions
$\{X^i_0\}_{i=1}^N$  are i.i.d., independent of $\{B^i_t\}$,
and drawn from the initial
distribution $p^*(x,0)$ of $X_0$.  Both  $\{B^i_t\}$ and $\{X^i_0\}$
are also assumed to be independent of $X_t,Z_t$.

As in~\Sec{sec:prelim}, we impose admissibility requirements on the
control input $U^i_t$ in~\eqref{eqn:particle_model}:
\begin{definition}[Admissible Input]
The control input $U^i_t$ is {\em admissible} if the random variables $u(x,t) $
and $\v(x,t)$ are  $\clZ_t = \sigma(Z_s:s\le t)$ measurable for each $t$.   Moreover,  each $t$,
\begin{romannum}

\item
$ \Expect[| u(X^i_t,t) |  +  | \v(X^i_t,t) |^2 ]<\infty$, and with probability one,
 \begin{subequations}
\begin{eqnarray}
\lim_{x\rightarrow\pm\infty} u(x,t) p(x,t)  &=& 0,
\label{e:adU}
\\
\lim_{x\rightarrow \pm \infty} \v(x,t) p(x,t) &=&  0.
\label{e:adV}
\end{eqnarray}
\end{subequations}
where $p$ is the posterior distribution  of $X^i_t$ given $\clZ_t$, defined in \eqref{e:pDef}.
\item
$u:\Re^2\rightarrow\Re$, $\v:\Re^2\rightarrow\Re$ are twice
continuously differentiable in their first arguments.

\spm{Is there an analog to,
$1+u_n'(x)$ is non-zero for all $x$. }

\end{romannum}
\qed
\end{definition}

The functions $\{ u(x,t),\v (x,t)\}$ represent the
continuous-time counterparts of the optimal control function
$v_n(x)$ (see~\eqref{eqn:optimiz_problem}).   We say that these
functions are \textit{optimal} if $p\equiv p^*$, where recall $p^*$ is the posterior distribution
of $X_t$ given $\clZ_t$ as defined in~\eqref{e:pXDef}.
Given $p^*(\cdot,0)= p(\cdot,0)$, our goal is to choose $\{u,\v
\}$ in the feedback particle filter so that the evolution
equations of these conditional distributions coincide.


The evolution of   $p^*(x,t)$ is described by the Kushner-Stratonovich (K-S) equation:
\spm{Too kiss-ass:  Recall that the celebrated}
\begin{equation}
\ud p^\ast = \clL^\dagger p^\ast \ud t + \frac{1}{\sigma_W^2}( h-\hat{h}_t )(\ud Z_t - \hat{h}_t \ud t)p^\ast, \label{eqn:Kushner_eqn}
\end{equation}
where $ \hat{h}_t = \int h(x) p^*(x,t) \ud x$, and $ \clL^\dagger
p^\ast = - \frac{\partial (p^\ast a)}{\partial
x}+\frac{\sigma_B^2}{2}\frac{\partial^2 p^\ast}{\partial x^2}
$.

The evolution equation of $p(x,t)$ is described next.
The proof appears in \Appendix{apdx:Derivation_FPK}.

\begin{proposition}
\label{thm:FPK}
Consider the process $X^i_t$ that evolves according to the particle
filter model~\eqref{eqn:particle_model}.  The conditional distribution
of $X^i_t$ given the filtration $\clZ_t$, $p(x,t)$, satisfies
the forward equation
\begin{equation}
\ud p = \clL^\dagger p \ud t  - \frac{\partial}{\partial x}\left( \v p \right) \ud Z_t - \frac{\partial}{\partial x}\left( u p
\right) \ud t + \sigma_W^2 \frac{1}{2}\frac{\partial^2}{\partial x^2}\left( p \v^2 \right) \ud t.
\label{eqn:mod_FPK}
\end{equation}
\qed
\end{proposition}

\subsection{Consistency with the Nonlinear Filter}
The main result of this section is the construction of an
optimal pair $\{u,\v\}$ under the following assumption:

\begin{romannum}
\item[\textbf{\textit{Assumption~A1}}]
The conditional distributions $(p^*, p)$ are $C^2$, with
  $p^*(x,t)>0$  and  $p(x,t)>0$, for all $x\in\Re$, $t>0$.
\qed
\end{romannum}

We henceforth choose  $\{u,\v\}$ as the solution to a certain
E-L BVP based on $p$: the function $\v$ is the solution to 
\begin{equation}
-\frac{\partial}{\partial x}\left( \frac{1}{p(x,t)} \frac{\partial
  }{\partial x} \{ p(x,t)\v(x,t) \}\right) = \frac{1}{\sigma_W^2} h'(x),
\label{eqn:EL_v*}
\end{equation}
with boundary condition~\eqref{e:adV}.
The function~$u(\cdot,t):\Re\to\Re$ is obtained as:
\begin{equation}
u(x,t) = \v(x,t) \left( -\frac{1}{2} (h(x) + \hat{h}_t ) + \frac{1}{2} \sigma_W^2
\v'(x,t) \right),
\label{eqn:u_intermsof_v*}
\end{equation}
where $ \hat{h}_t = \int h(x) p(x,t) \ud x$. We assume moreover
that the control input obtained using
$\{u,\v\}$ is admissible. The particular form of $u$ given
in~\eqref{eqn:u_intermsof_v*} and the BVP~\eqref{eqn:EL_v*} is
motivated by considering the continuous-time limit
of~\eqref{eqn:EL_eqn}, obtained on
letting $ \sTriangle := t_{n+1}- t_{n}$ go to zero; the calculations appear
in \Appendix{apdx:EL_uv}.

Existence and uniqueness of $\{u,\v\}$ is obtained in the
following proposition --- Its proof is given  in
\Appendix{apdx:uniqueness}.
\begin{proposition}
\label{prop:existence_uniqueness_properties_EL}
Consider the BVP~\eqref{eqn:EL_v*},   subject to Assumption~A1.  Then,
\begin{enum}
\item There exists a unique solution $\v$, subject to the
    boundary condition \eqref{e:adV}.

\item The solution satisfies $\v(x,t) \ge 0$ for all $x,t$,
    provided $h'(x)\ge 0$ for all $x$. \qed
\end{enum}
\end{proposition}



The following theorem shows that the two evolution equations
\eqref{eqn:Kushner_eqn} and \eqref{eqn:mod_FPK} are identical.
The proof appears in \Appendix{apdx:consistency_pf}.

\begin{theorem}
\label{thm:kushner}
Consider the two evolution equations for $p$ and $p^*$, defined
according to the solution of the forward
equation~\eqref{eqn:mod_FPK} and the K-S
equation~\eqref{eqn:Kushner_eqn}, respectively.  Suppose that
the control functions $u(x,t)$ and $\v(x,t)$ are obtained
according to~\eqref{eqn:EL_v*} and~\eqref{eqn:u_intermsof_v*},
respectively. Then, provided $p(x,0)=p^*(x,0)$, we have for all
$t\ge 0$,
\[
p(x,t) = p^*(x,t)
\]
\qed
\end{theorem}

\begin{remark}
 \Theorem{thm:kushner} is based on the \textit{ideal setting} in which the gain
$\v(X^i_t,t) $ is obtained as a function of the posterior $p=p^*$ for
$X^i_t$.   In practice the algorithm is applied with $p$ replaced by
the empirical distribution of the $N$ particles.  

In this ideal setting, the empirical distribution of the particle system will approximate the posterior distribution $p^{*}(x,t)$ as
$N\rightarrow \infty$.   The convergence is in the weak sense in
general.  To obtain almost sure convergence, it is necessary to obtain
sample path representations of the solution to the stochastic
differential equation for each $i$ (see e.g.~\cite{kun90}).  Under these conditions the solution to the SDE \eqref{eqn:particle_filter} for each $i$ has a functional representation,
\[
X^i_t = F(X^i_0, B^i_{[0,t]}; Z_{[0,t]} ),
\]
where the notation $Z_{[0,t]}$ signifies the entire sample path $\{Z_s
: 0\le s\le t\}$ for a stochastic process $\bfmZ$; $F$ is a continuous functional (in the uniform topology) of the sample paths $ \{B^i_{[0,t]} ,Z_{[0,t]} \}$ along with the initial condition $X^i_0$.
It follows that the empirical distribution has a functional representation,
\[
p^{(N)}(A,t) = \frac{1}{N}\sum_{i=1}^N  \ind\{ F(X^i_0, B^i_{[0,t]}; Z_{[0,t]} ) \in A\}
\]
The sequence $\{(X^i_0, B^i_{[0,t]}) : i=1,...\}$ is i.i.d.\ and
independent of $\bfmZ$. It follows that the summand   $\{ \ind\{ F(X^i_0, B^i_{[0,t]}; Z_{[0,t]} ) : i=1, \dots\}$ is also i.i.d.\ given $Z_{[0,t]}$.   Almost sure convergence follows from the Law of Large Numbers for scalar i.i.d.\ sequences.   

In current research we are considering the more
difficult problem of performance bounds for the approximate
implementations described in \Sec{sec:nonlinear}.
\spm{ok?}
\end{remark}



\begin{remark}
On integrating~\eqref{eqn:EL_v*} once, we obtain an {\em equivalent}
characterization of the E-L BVP:
\begin{equation}
\frac{\partial }{\partial x}(p \v) =
-\frac{1}{\sigma_W^2}(h-\hat{h}_t)p,
\label{eqn:EL_v_first_order}
\end{equation}
now with a single boundary condition
$\lim_{x\rightarrow-\infty} p\v(x,t) =0$.  The resulting gain function
can be readily shown to yield admissible control input under
certain additional technical assumptions on density $p$ and the
function $h$.

Given the scope of this paper, and the fact that the same apriori
bounds apply also to the multivariable case, we defer additional
discussion to a future publication.

\spm{perhaps too much advertsing regarding future work?}
\spm{this seems misplaced:
  In the following, we
provide an admissible control law for the linear Gaussian case.}
\end{remark}

\begin{remark}
Although the methodology and the filter is presented for Gaussian
process and observation noise, the case of non-Gaussian process noise
is easily handled -- simply replace the noise model in the filter with
the appropriate model of the process noise.

\spm{not clear enough.  Is it necessary to keep it?
The assumption on the observation noise is more restrictive.  Indeed,
the filter is no longer expected to be consistent for more general
models of the observation noise.
}

For other types of observation noise, one would modify the conditional
distribution $p_{Y|X}$ in the optimization
problem~\eqref{eqn:optimiz_problem}.  The derivation of filter would
then proceed by consideration of the first variation (see \Appendix{apdx:EL_uv}).


\spm{This is really unclear -- I don't see the value, and it distracts from the previous valuable comment:
The feedback particle filter, in effect, serves to illustrate the
methodology for the most well-studied case of Gaussian observation
noise.}
\end{remark}



\subsection{Example: Linear Model}
\label{sec:linear_case}

It is helpful to consider the feedback particle filter in the following simple linear setting, 
\begin{subequations}
\begin{align}
\ud X_t  &= \alpha \;X_t\ud t + \sigma_B \ud B_t,\label{eqn:dyn_lin}\\
\ud Z_t &= \gamma\; X_t \ud t+\sigma_W \ud W_t,\label{eqn:obs_lin}
\end{align}
\end{subequations}
where $\alpha$, $\gamma$ are real numbers. 
The initial
distribution $p^*(x,0)$ is assumed to be
Gaussian with mean $\mu_0$ and
variance $\Sigma_0$.

The following lemma provides the solution of the gain function
$\v(x,t)$ in the linear Gaussian case.

\begin{lemma}\label{lem_2}
Consider the linear observation equation~\eqref{eqn:obs_lin}.
If  $p(x,t)$
is assumed to be Gaussian with mean $\mu_t$
and variance $\Sigma_t$,
then the solution of E-L BVP~\eqref{eqn:EL_v_intro}
is given by:
\begin{align}
\v(x,t) = \frac{\Sigma_t \gamma}{\sigma_W^2}.
\label{eqn:linsol_v}
\end{align}
\qed
\end{lemma}

The formula~\eqref{eqn:linsol_v} is verified by direct
substitution in the ODE~\eqref{eqn:EL_v_intro} where the
distribution $p$ is Gaussian.

The optimal control yields the following form for the particle
filter in this linear Gaussian model:
\begin{equation}
\begin{aligned}
\ud X^i_t= \alpha \; X^i_t \ud t +
\sigma_B \ud B^i_t + \frac{ \Sigma_t \gamma}{\sigma^2_W} \left( \ud Z_t - \gamma \frac{X^i_t + \mu_t}{2} \ud t \right).
\end{aligned}
\label{eqn:particle_filter_lin}
\end{equation}

Now we show that $p=p^*$ in this case. That is, the conditional
distributions of $\bfmX$ and $\bfmX^i$ coincide, and are
defined by the well-known dynamic equations that characterize
the mean and the variance of the continuous-time Kalman filter.
The proof appears in \Appendix{pf_thm1}.

\begin{theorem}
\label{thm_lin} Consider the linear Gaussian filtering problem
defined by the state-observation equations
(\ref{eqn:dyn_lin},\ref{eqn:obs_lin}).  In this case the
posterior distributions of $\bfmX$ and $\bfmX^i$  are Gaussian,
whose conditional mean and covariance are given by the
respective SDE and the ODE,
\begin{align}
 \ud \mu_t  &= \alpha\mu_t \ud t + \frac{\Sigma_t \gamma }{\sigma^2_W} \Bigl(\ud Z_t-\gamma\mu_t \ud t \Bigr)
 \label{eqn:mod1}
\\[.1cm]
\ddt \Sigma_t
&= 2\alpha \Sigma_t + \sigma_B ^2 -\frac{(\gamma)^2 \Sigma^2_t}{\sigma^2_W}
\label{eqn:mod2}
\end{align}
\qed
\end{theorem}


Notice that the particle system~\eqref{eqn:particle_filter_lin} is
not practical since it requires computation of the conditional
mean and variance $\{\mu_t, \Sigma_t\}$. If we are to compute
these quantities, then there is no reason to run a particle
filter!

In practice $\{\mu_t, \Sigma_t\}$ are approximated as sample means and sample covariances from the ensemble
$\{X^i_t\}_{i=1}^N$:
\begin{equation}
\begin{aligned}
\mu_t & \approx \mu_t^{(N)} := \frac{1}{N} \sum_{i=1}^N X^i_t,\\
\Sigma_t & \approx \Sigma_t^{(N)} := \frac{1}{N-1} \sum_{i=1}^N (X^i_t - \mu_t^{(N)})^2.
\end{aligned}
\label{e:mut_sigmat_approx}
\end{equation}
The resulting equation
\eqref{eqn:particle_filter_lin} for the $i^{\text{th}}$
particle is given by
\begin{equation}
\begin{aligned}
\ud X^i_t= \alpha \; X^i_t \ud t +
\sigma_B \ud B^i_t + \frac{\Sigma_t^{(N)}\gamma}{\sigma^2_W} \left( \ud Z_t - \gamma \frac{X^i_t + \mu_t^{(N)}}{2} \ud t \right).
\end{aligned}
\label{eqn:particle_filter_lin_implement}
\end{equation}
It is very similar to the mean-field ``synchronization-type'' control laws and oblivious equilibria
constructions as in \cite{huacaimal07,YinMehtaMeynShanbhag12TAC}.
The model~\eqref{eqn:particle_filter_lin} represents the mean-field
approximation obtained by letting $N\rightarrow \infty$.


\subsection{Feedback Particle Filter for the Multivariable Model}
\label{sec:FPF_multi-variable}

Consider the model 
\eqref{eqn:Signal_Process}-\eqref{eqn:Obs_Process} in which the state
$X_t$ is $d$-dimensional,  with $d\ge 2$, so that $a(\cdot)$ is   a
vector-field on $\Re^d$. For  ease of presentation $\sigma_B$ is
assumed to be scalar, and the observation process $Z_t \in \Re$
real-valued.  

To aid comparison with Crisan and Xiong's work, we express the feedback particle filter in its Stratonovich form:
\begin{equation}
\ud X^i_t = a ( X^i_t) \ud t + \sigma_B \ud B^i_t  + \v(X^i_t,t) \circ
\ud \Inov^i_t 
\label{eqn:particle_filter_nonlin_multi}
\end{equation}
where the innovation error is as before,
\begin{equation}
\ud \Inov^i_t \eqdef \ud Z_t - \frac{1}{2}    (h(X^i_t) + \hat{h}_t) \ud t,
\label{e:in_multi}
\end{equation}
and the gain function $\v(x,t) = \left(\v_1,\v_2,...,\v_d\right)^T$ is now vector-valued.
It is given by the solution of a BVP,
the multivariable counterpart of
\eqref{eqn:EL_v_first_order}:
\begin{equation}
\nabla \cdot (p \v) = - \frac{1}{\sigma_W^2} (h-\hat{h}_t)p,
\label{e:bvp_divergence_multi}
\end{equation}
where $\nabla \cdot$ denotes the divergence operator.

It is straightforward to prove consistency by repeating the steps in
the proof of \Theorem{thm:kushner}, now with the Kolmogorov forward
operator:
\begin{equation}
\ud p = \clL^\dagger p \ud t  - \nabla \cdot \left( \v p
\right) \ud Z_t  - \nabla \cdot \left( u p
\right) \ud t + \frac{1}{2} \sigma_W^2 \sum_{i,j=1}^d \frac{\partial^2 [(\v \v^T)_{ij} p]}{\partial x_i \partial x_j} \ud t
\label{eqn:mod_FPK_multi}
\end{equation}
where $ \clL^\dagger p = - \nabla \cdot (p a) + \frac{1}{2} \sigma_B^2
\Lap p$, $\Lap$ is the Laplacian, and $u$ is the multivariable
counterpart of~\eqref{eqn:u_intermsof_v*}:
\[
	u = - \v(x,t) \frac{h(x) + \hat{h}_t}{2} + \w(x,t),
\]
where $\w = \left(\w_1,\w_2,...,\w_d\right)^T$ is the Wong-Zakai correction term:
\begin{equation*}
\w_l(x,t) := \frac{1}{2} \sigma_W^2 \sum_{k=1}^d \v_{k}(x,t)  \frac{\partial \v_{l}}{\partial x_{k}}(x,t).
\label{eqn:wong_term_intro}
\end{equation*}

As with the scalar case, the multivariable feedback particle filter
requires solution of a BVP~\eqref{e:bvp_divergence_multi} at each time
step.  

\spm{I really did not like the remaining text since we are pretending that C+X's representation makes sense!
I've   changed the tone}

Following the work of Crisan and Xiong~\cite{CriXiong09}, we might assume the following representation in an attempt to solve~\eqref{e:bvp_divergence_multi},
\begin{equation}
p \v = \nabla \phi.
\label{eqn:pv_is_gradient}
\end{equation}
where $\phi$ is assumed to be sufficiently smooth.  Substituting~\eqref{eqn:pv_is_gradient} in~\eqref{e:bvp_divergence_multi} yields the Poisson equation,
\begin{equation}
\Lap \phi = - \frac{1}{\sigma_W^2} (h-\hat{h}_t)p\, .
\label{eqn:Poisson}
\end{equation}

A solution to Poisson's equation with $d\ge 2$ can be expressed in terms of 
Green's function: 
\[
G(r) = \begin{cases}
\frac{1}{2\pi} \ln(r) & \text{ for $d=2$;}
\\
\frac{1}{d(2-d)\omega_d}r^{2-d} &\text{for $d>2$ ,}
\end{cases}
\]
where $\omega_d$ is
the volume of the unit ball in $\Re^d$. 
A solution to \eqref{eqn:Poisson} is then given by,
\begin{equation*}
\phi(x) = - \frac{1}{\sigma_W^2} \int_{\Re^d} G(|y-x|) (h(y)-\hat{h}_t) p(y,t) \ud y,
\end{equation*}
where $|y-x| := \left( \sum_{j=1}^d (y_j-x_j)^2
\right)^{\frac{1}{2}}$ is the Euclidean distance.

On taking the gradient and
using~\eqref{eqn:pv_is_gradient}, one obtains an explicit
formula for the gain function:
\begin{equation}
\v(x,t) = \frac{1}{\sigma_W^2} \frac{\Gamma(p,h)(x)}{p(x,t)} =: \v_g(x,t),
\label{e:integral_form_of_solution_multi}
\end{equation}
where $\Gamma(p,g)(x) := \frac{1}{d\omega_d} \int \frac{y-x}{|y-x|^{d}} (g(y)-\hat{g})
p(y) \ud y$.

While this leads to a solution to~\eqref{e:bvp_divergence_multi}, it may not lead to an admissible control law.  This difficulty arises in the prior work of Crisan and Xiong.

\subsection{Comparison with Crisan and Xiong's Filter}
\label{sec:comparison}

In~\cite{CriXiong09} and in Sec.~4 of~\cite{Xiong}, a particle filter of the following form is presented:
\begin{equation}
\ud {X}^i_t = {a}({X}^i_t) \ud t +   \sigma_B \ud {B}^i_t +
\v_g({X}^i_t,t) \Bigl( \frac{\ud}{\ud t}\tilde{I}_t \Bigr)
\ud t,
\label{filter_Crisan_1}
\end{equation}
where $\tilde{I}_t$ is a
certain smooth approximation obtained from the standard form of the innovation error
$ I_t := Z_t - \int_0^t \hat{h}_t \ud t $.  A consistency result is described for
this filter.

We make the following comparisons:
\begin{enum}
\item Without taking a smooth approximation, the filter~\eqref{filter_Crisan_1} is formally equivalent to the following SDE expressed here in its Stratonovich form:
\begin{equation}
\ud {X}^i_t = {a}({X}^i_t) \ud t +   \sigma_B \ud {B}^i_t +
\v_g({X}^i_t,t) \circ \Bigl(\ud Z_t - \hat{h}_t \ud t\Bigr).
\label{filter_Crisan_1_SDE}
\end{equation}
In this case, using~\eqref{eqn:mod_FPK_multi}, it is straightforward
to show that the consistency result 
\textit{does not hold}.  
\spm{NOTE ITALIC!!}
In particular,
there is an extra second order term that is not present in the K-S
equation for evolution of the true posterior $p^*(x,t)$.

\item The feedback particle filter introduced in this paper
does not require a smooth approximation and yet achieves
  consistency. The key breakthrough is the modified definition of the
  innovation error (compare~\eqref{filter_Crisan_1_SDE} with~\eqref{eqn:particle_filter_nonlin_multi}).  Note that the innovation
  error~\eqref{e:in_multi} is not assumed apriori but comes about via
  analysis of the variational problem.  This is one utility
of introducing the variational formulation.  Once the feedback particle filter
  has been derived, it is straightforward to prove consistency (see
  the Proof of \Theorem{thm:kushner}). 
  
  \spm{No value here: And
  the proof in turn represents a straightforward derivation of the
  feedback particle filter in the multivariable case.}

\item \spm{I don't get this lead-in: For the form described in this section, }
The computational
  overhead for the feedback particle filter and the filter of Crisan and Xiong are equal.  
  Both require the approximation of the
  integral~\eqref{e:integral_form_of_solution_multi}, and division by (a
  suitable regularized approximation of) $p(x,t)$.  Numerically, the Poisson equation
  formulation~\eqref{eqn:Poisson} of the E-L  BVP~\eqref{e:bvp_divergence_multi} is convenient.  There
  exist efficient numerical algorithms to approximate the
  integral solution~\eqref{e:integral_form_of_solution_multi} for a system of
  $N$ particles in arbitrary dimension; cf.,~\cite{Greengard:1987}.
\end{enum}


However, while appealing, the function $\v_g$ is {\em not} the correct form of
the gain function in the multivariable case,  even for linear models:
It is straightforward to verify that the
Kalman gain is a solution of the boundary value
problem~\eqref{e:bvp_divergence_multi}.  Using the Kalman gain for the
gain function in~\eqref{eqn:particle_filter_nonlin_multi} yields the feedback particle filter for the
multivariable linear Gaussian case.  The filter is a
straightforward extension of~\eqref{eqn:particle_filter_lin_implement}
in \Sec{sec:linear_case}.

However, the Kalman gain solution is not of the form~\eqref{eqn:pv_is_gradient}.
Thus, the integral solution~\eqref{e:integral_form_of_solution_multi} does {\em not} equal the Kalman gain in the linear Gaussian case (for $d\ge 2$). 

Moreover,   the gradient form solution is unbounded:
$|\v_g(x,t)|\rightarrow\infty$ as $|x|\rightarrow\infty$, and 
		${\sf E}[|\v_g|] = {\sf E}[|\v_g|^2] = \infty$.  
A proof is given in
\Appendix{appdx:sol_multi}.

It follows that the control input
obtained using $\v_g$ is not admissible, and hence the Kolmogorov
forward operator is no longer valid.  Filter implementations using
$\v_g$ suffer from numerical issues on account of large unbounded
gains.  In contrast, the feedback particle filter using Kalman gain
works both in theory and in practice.

\spm{I don't think this adds much here -- we say this in so many places:
It is also worthwhile to note that {\em both} the gain function and
the innovation error are important.  For example, even if one were to
replace $\v_g$ by the Kalman gain in~\eqref{filter_Crisan_1_SDE}, the
resulting filter will still not be exact in the linear Gaussian
case.  For the filter to be exact, one requires the modified form of the
innovation error as given in formula~\eqref{e:in_multi}.}

The choice of gain function in the multivariable case requires careful
consideration of the uniqueness of the solutions of the BVP:
The  solution of \eqref{e:bvp_divergence_multi} is not unique, even though uniqueness holds when 
$p\v$ is assumed to be of a gradient-form.  \spm{is it unique?}
\spm{huh??
\\
It is of
interest to obtain gain function solution(s) that yield admissible
control input.}
\spm{been said:
Additional details on the multivariable
feedback particle filter are planned for future publication.
}

Before closing this section, we note that 
\cite[Proposition~2.4]{CriXiong09} concerns another filter 
that does not rely on smooth approximation,
\[
\ud {X}^i_t = {a}({X}^i_t) \ud t +   \sigma_B \ud {B}^i_t +
\v_g({X}^i_t,t) \circ \ud Z_t - \frac{1}{\sigma^2_W} \frac{\Gamma(p,\frac{1}{2}|h|^2)(X_t^i)}{p(X_t^i,t)} \ud t.
\]
Our calculations indicate that consistency is also an issue  for this
filter.  The issue with $\v_g$ also applies to this filter.   A more
complete comparison needs further investigation.


\section{Synthesis of the Gain Function}
\label{sec:nonlinear}

Implementation of the nonlinear
filter~\eqref{eqn:particle_filter_nonlin_intro} requires
solution of the E-L BVP~\eqref{eqn:EL_v_intro} to obtain the
gain function $\v(x,t)$ for each fixed $t$.


\subsection{Direct Numerical Approximation of the BVP solution}

The explicit closed-form
formula~\eqref{eqn:closed_form_solution_of_BVP} for the solution of the
BVP~\eqref{eqn:EL_v_intro} can be used to construct a direct numerical
approximation of the solution.
Using~\eqref{eqn:closed_form_solution_of_BVP}, we have
\begin{equation*}
\v(x,t) = \frac{1}{p(x,t)} \frac{1}{\sigma_W^2} \int_{-\infty}^x
(\hat{h}_t - h(y)) p(y,t) \ud y.
\end{equation*}

The approximation involves three steps:
\begin{enum}
\item Approximation of $\hat{h}_t$ by using a sample mean:
\begin{equation*}
\hat{h}_t \approx \frac{1}{N} \sum_{j=1}^N h(X_t^j) =: \hat{h}_t^{(N)}.
\end{equation*}
\item Approximation of the integrand:
\begin{equation*}
(\hat{h}_t - h(y)) p(y,t) \approx \frac{1}{N} \sum_{j=1}^N
(\hat{h}_t^{(N)} - h(X_t^j)) \delta(y-X_t^j),
\end{equation*}
where $\delta(\cdot)$ is the Dirac delta function.
\item Approximation of the density $p(x,t)$ in the denominator,
  e.g., as a sum of Gaussian:
\begin{equation}
p(x,t) \approx  \frac{1}{N} \sum_{j=1}^N q^j_t(x) =: \tilde{p}(x,t),
\label{eqn:sum_of_Gaussian_for_approx_numeric}
\end{equation}
where $q^j_t(x) = q(x;X^j_t,\varepsilon)= \frac{1}{\sqrt{2\pi\varepsilon}}
\exp\left(-\frac{1}{2\varepsilon}(x-X^j_t)^2\right)$.  The appropriate value of
$\varepsilon$ depends upon the problem.  As a function of
$N$, $\varepsilon$ can be made smaller as $N$ grows; As $N\rightarrow \infty$, $\varepsilon \rightarrow 0$.
\end{enum}

This yields the following numerical approximation of the gain
function:
\begin{equation}
\v(x,t) = \frac{1}{\tilde{p}(x,t)} \frac{1}{\sigma_W^2}   \frac{1}{N} \sum_{j=1}^N (\hat{h}_t^{(N)} - h(X^j_t)) H(x-X^j_t),
\label{eqn:soln_w_sum_of_delta}
\end{equation}
where $H(\cdot)$ is the Heaviside function.

Note that the gain function needs to be evaluated only at the particle
locations $X^i_t$.  An efficient $O(N^2)$ algorithm is easily
constructed to do the same:
\begin{eqnarray}
\v(X^i_t,t) &=& \frac{1}{\tilde{p}(X^i_t,t)} \frac{1}{\sigma_W^2}
  \frac{1}{N} \left( \sum_{j: X^j_t<X^i_t} (\hat{h}_t^{(N)} - h(X^j_t)) +
\frac{1}{2} (\hat{h}_t^{(N)} - h(X^i_t)) \right)\nonumber\\
\v'(X^i_t,t) &=& \frac{1}{\sigma_W^2}(\hat{h}_t^{(N)} - h(X^i_t)) - \tilde{b}(X^i_t) \v(X^i_t,t),
\label{eqn:soln_w_sum_of_delta_at_particles}
\end{eqnarray}
where $\tilde{b}(x):=\frac{\partial}{\partial x}(\ln
\tilde{p})(x,t)$.  For $\tilde{p}$ defined using
the sum of Gaussian approximation~\eqref{eqn:sum_of_Gaussian_for_approx_numeric}, a
closed-form formula for $\tilde{b}(x)$ is easily obtained.

 \spm{moved consistency discussion}

\spm{Summary of the?  Haven't we summarized enough?  Is this section heading ok?
\\
And, where did "DNS" come from???}

\def\Algo#1{Alg.~\ref{#1}}

\subsection{Algorithm}
\label{sec:Alg}

For implementation purposes, we use the Stratonovich form of the
filter (see~\eqref{eqn:FPF_Strato}) together with an Euler
discretization.  The resulting discrete-time algorithm appears in Algorithm~1. 
At each time step, the algorithm requires approximation of the gain
function.  A DNS-based algorithm for the same is summarized in Algorithm~2.

\spm{I had trouble with the algorithm citation -- is it ok now?}

In practice, one can use a less computationally intensive algorithm to
approximate the gain function.  An algorithm based on sum-of-Gaussian approximation of density appears in our conference paper~\cite{YangMehtaMeyn_cdc11}.  In
the application example presented in \Sec{sub2}, the gain function is approximated by using Fourier series.

\begin{algorithm}
\caption{Implementation of feedback particle filter}
\begin{algorithmic}[1]
\STATE{\bf Initialization}
\FOR{$i:=1$ to $N$}
\STATE Sample $X_0^{i}$ from $p(x,0)$
\ENDFOR
\STATE Assign value $t:=0$
\end{algorithmic}

\begin{algorithmic}[1]
\STATE{\bf Iteration}  $\;$  [from $t$ to $t+\sTriangle t$]
\STATE Calculate $\hat{h}_t^{(N)} := \frac{1}{N}\sum_{i=1}^N h(X_t^{i})$
\FOR{$i:=1$ to $N$}
\STATE Generate a sample, $\sTriangle V$, from $N(0,1)$
\STATE Calculate $\sTriangle I_t^{i} := \sTriangle Z_t - \frac{1}{2} \left( h(X_t^i) + \hat{h}^{(N)}_t\right)\sTriangle t$
\STATE Calculate the gain function $\v(X_t^i,t)$ (e.g., by using \Algo{alg:Gain-DNS})  
\spm{removed " other algorithm"}
\STATE $X_{t+\sTriangle t}^{i} := X_t^{i} + a(X_t^{i}) \sTriangle t + \sigma_B\sqrt{\sTriangle t} \sTriangle V+ \v(X_t^i,t) \sTriangle I_t^{i}$
\spm{I added the sigma term - ok?}
\ENDFOR
\STATE $t :=  t+\sTriangle t$
\spm{not $p^*$}
\end{algorithmic}

\label{alg:discrete-FPF}
\end{algorithm}

\begin{algorithm}
\caption{Synthesis of gain function $\v(x,t)$}

\begin{algorithmic}[1]
\STATE Calculate $\hat{h}_t \approx \hat{h}_t^{(N)};$
\STATE Approximate $p(x,t)$ as a sum of Gaussian:
\begin{equation}
p(x,t) \approx \tilde{p}(x,t) := \frac{1}{N} \sum_{j=1}^N q_t^j(x),\nonumber
\end{equation}
where $q_t^j (x) =\frac{1}{\sqrt{2\pi\varepsilon}} \exp\left(-\frac{(x-X_t^j)^2}{2\varepsilon}\right)$. 
\STATE Calculate the gain function
\begin{equation}
\v(x,t) := \frac{1}{\tilde{p}(x,t)}\frac{1}{\sigma_W^2} \frac{1}{N} \sum_{j=1}^N \left(\hat{h}_t^{(N)} - h(X_t^j)\right) H(x-X_t^j),\nonumber
\end{equation}
where $H(\cdot)$ is the Heaviside function.
\end{algorithmic}

\label{alg:Gain-DNS}

\end{algorithm}

\spm{why are these further remarks squeezed in here?}

\subsection{Further Remarks on the BVP}

Recall that the solution of the nonlinear filtering problem is given
by the Kushner-Stratonovich nonlinear evolution PDE.  The feedback
particle filter instead requires, at each time $t$, a solution of the
linear BVP~\eqref{e:bvp_divergence_multi} to obtain the gain function $\v$:
\begin{equation*}
\nabla \cdot (p \v) = - \frac{1}{\sigma_W^2} (h-\hat{h}_t)p.
\end{equation*}

We make the following remarks:
\begin{enum}
\item There are close parallels between the proposed algorithm and the
  vortex element method (VEM) developed by Chorin and others for solution of the
    Navier-Stokes evolution PDE; cf.,~\cite{Chorin73,Leonard80}.
    In VEM, as in the feedback particle filter, one obtains the
    solution of a nonlinear evolution PDE by flowing a large number of
    particles.  The vector-field for the particles is obtained by
    solving a linear BVP at each time.

Algorithms based on VEM are popular in the large Reynolds number
regime when the domain is not too complicated.  The latter requirement
is necessary to obtain solution of the linear BVP in tractable
fashion~\cite{Greengard:1987}.

\item One may ask what is the benefit, in terms of accuracy and
  computational cost, of the feedback particle filter-based solution
  when compared to a direct solution of the nonlinear PDE (Kushner-Stratonovich equation) or the linear PDE (Zakai equation)?

The key point, we believe, is robustness on account of the feedback
control structure.  Specifically, the self-correcting property of the
feedback provides robustness, allowing one to tolerate a degree of
uncertainty inherent in any model or approximation scheme.  This is
expected to yield accurate solutions in a computationally efficient
manner.
A complete answer will require further analysis, and as such
reflects an important future direction.

\item
The biggest computational cost of our approach is the need to solve the BVP at
each time-step, that additionally requires one to approximate the
density.  We are encouraged however by the extensive set of
tools in feedback
control: after all, one rarely needs to solve the HJB equations in
closed-form to obtain a reasonable feedback control law. 
Moreover, there
are many approaches in nonlinear and adaptive control to both
approximate control laws as well as learn/adapt these in online
fashion; cf.,~\cite{bertsi96a}.

\spm{The Daum citation below should be moved up!
It does not belong right after the neuro DP citation.
}

\end{enum}

\section{Numerics}
\label{sec:numerics}

\begin{figure*}[t]
\begin{center}
\Ebox{1}{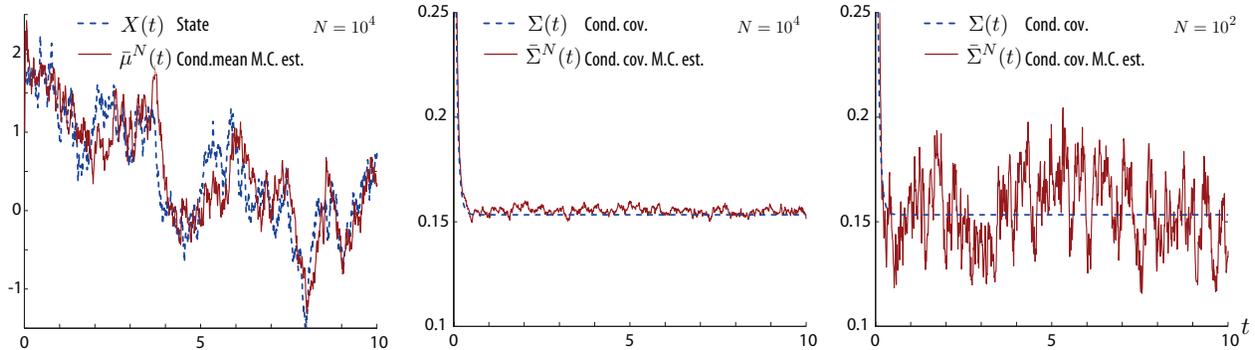}
\end{center}
\vspace{-0.2in}
\caption{(a) Comparison of the true state $\{X_t\}$ and the
conditional mean $\{\bar{\mu}^{(N)}_t\}$. (b) and (c) Plots of
estimated conditional covariance with $N=10,000$ and $N=100$
particles, respectively.  For comparison, the true conditional
covariance obtained using Kalman filtering equations is also
shown.}
						\vspace{-.25cm}
\label{fig_lin}
\end{figure*}

\subsection{Linear Gaussian Case}

Consider the linear system:
\begin{subequations}
\begin{align}
\ud X_t  &= \alpha \;X_t\ud t + \ud B_t,\label{eqn:dyn_lin_a}\\
\ud Z_t &= \gamma\; X_t \ud t+\sigma_W \ud W_t, \qquad X_0\sim N(1,1)\,,\label{eqn:obs_lin_a}
\end{align}
\end{subequations}
where $\{B_t\},\{W_t\}$ are mutually independent standard
Wiener process, and parameters $\alpha=-0.5$, $\gamma=3$ and
$\sigma_W=0.5$. 
\spm{problem with the figure - the mean of $X_0$ does not appear to be one!}

Each of the $N$ particles is described by the linear SDE, 
\begin{equation}
\ud X^i_t = \alpha\;X^i_t \ud t +  \ud B^i_t +
\frac{\gamma \;\bar{\Sigma}^{(N)}_t}{\sigma_W^2}[\ud Z_t - \gamma\frac{X^i_t
  +\bar{\mu}_t^{(N)}}{2}\ud t]\, ,
\label{eqn:pf_lgeg}
\end{equation}
where  $\{B^i_t\}$ are mutually independent standard Wiener
process;  the particle system is initialized by drawing initial
conditions $\{X^i_0\}_{i=1}^N$ from the distribution
$\textbf{\emph{N}}(1,1)$, and the parameter values are chosen
according to the model.

In the simulation discussed next, the
mean $\bar{\mu}_t^{(N)}$ and the variance $\bar{\Sigma}_t^{(N)}$
are obtained from the ensemble $\{X^i_t\}_{i=1}^N$
according to~\eqref{e:mut_sigmat_approx}.

\Fig{fig_lin} summarizes some of the results of the
numerical experiments: Part (a) depicts a sample path of the
state $\{X_t\}$ and the mean $\{\bar{\mu}^{(N)}_t\}$ obtained
using a particle filter with $N=10,000$ particles. Part (b)
provides a comparison between the estimated variance
$\bar{\Sigma}^{(N)}_t$ and the true error variance $\Sigma_t$
that one would obtain by using the Kalman filtering equations.
The accuracy of the results is sensitive to the number of
particles. For example, part (c) of the figure provides a
comparison of the variance with $N=100$ particles.

\emph{Comparison with the bootstrap filter}: We next provide a performance
comparison between the feedback particle filter and the
bootstrap particle filter for the linear
problem~(\ref{eqn:dyn_lin_a}, \ref{eqn:obs_lin_a}) in  regard 
to both error and running time.

For the linear filtering problem, the optimal solution is given by the
Kalman filter.  We use this solution to define the relative mean-squared error:
\begin{equation}
mse = \frac{1}{T}\int_0^T \left(\frac{\Sigma_t^{(N)} - \Sigma_t}{\Sigma_t}\right)^2 \ud t,
\end{equation}
where $\Sigma_t$ is the error covariance
using the Kalman filter, and $\Sigma^{(N)}_t$
is its approximation using the particle filter.

\Fig{fig:fig_comp}(a) depicts a comparison between {\em mse}
obtained using the feedback particle
filter~\eqref{eqn:pf_lgeg} and the bootstrap filter.  The
latter implementation is based on an algorithm taken from Ch.~9
of~\cite{baincrisan07}.  For simulation purposes, we used a range of
values of $\alpha\in\{-0.5,0,0.5\}$, $\gamma=3$, $\sigma_B=1$, $\sigma_W=0.5$,
$\sTriangle t=0.01$, and $T = 50$.  
The plot is generated using simulations with $N=20, 50, 100, 200, 500, 1000$
particles.
\spm{I have to clean up these figures!}

\begin{figure}
\begin{center}
\Ebox{}{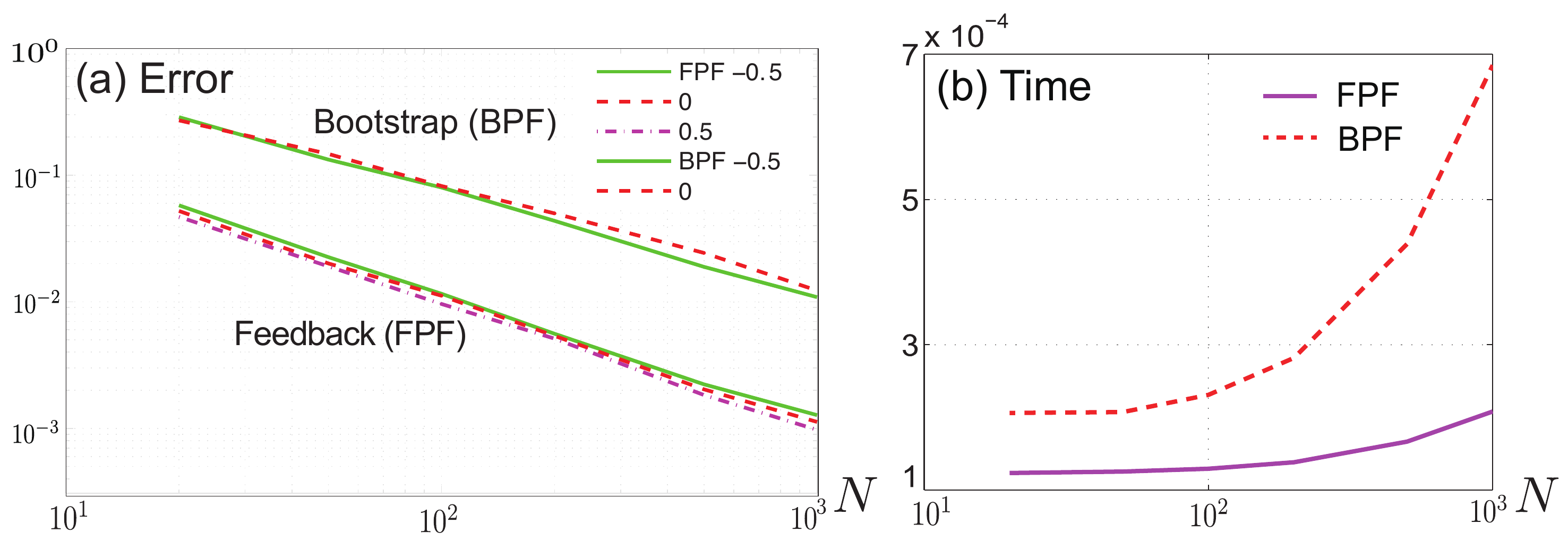}
\end{center}
\caption{Comparison of (a) the {\em mse}, (b) the computational time using feedback particle filter
  and the bootstrap particle filter.}
						\vspace{-.25cm}
\label{fig:fig_comp}
\end{figure}

\spm{I do not understand this paragraph. Are you saying we don't blow up?  It isn't clear.}
These numerical results suggest that feedback can help reduce the high
variance that is sometimes observed with the conventional particle
filter.  The variance issue can be especially severe
if the signal process~\eqref{eqn:Signal_Process} is unstable (e.g.,
$\alpha>0$ in~\eqref{eqn:dyn_lin_a}).
In this case, individual particles can exhibit numerical
instabilities due to time-discretization, floating point
representation etc.  With $\alpha>0$, our numerical
simulations with the bootstrap filter ``blew-up'' (similar conclusions were also arrived
independently in~\cite{DaumHuang10}) while the feedback particle filter is
provably stable based on observability of the
model~\eqref{eqn:dyn_lin_a}-\eqref{eqn:obs_lin_a} (see
\Fig{fig:fig_comp} mse plot with $\alpha=0.5$).

\Fig{fig:fig_comp}(b) depicts a comparison between computational time for the two filtering algorithms on the same problem. 
The time is given in terms of computation time per iteration cycle
(see Algorithm~1 in~\Sec{sec:Alg}) averaged over 100 trials. For simulation
purpose, we use MATLAB R2011b (7.13.0.564) on a 2.66GHz iMac as our test platform.


These numerical results suggest that, for a linear Gaussian
implementation, feedback particle filter has a lower computational cost
compared to the conventional bootstrap particle filter. The main
reason is that the feedback particle filter avoids the computationally
expensive resampling procedure. 

We also carried out simulations where the gain function is
approximated using Algorithm~2. In this case, the mse of the filter is
comparable to the mse depicted in \Fig{fig:fig_comp}(a).  However, the
computation time is larger than the bootstrap particle filter.  This
is primarily on account of the evaluation of the exponentials in
computing $\tilde{p}(x,t)$.  Detailed comparisons between the feedback
particle filter and the bootstrap particle filter will appear
elsewhere.

In general, the main computational burden of the
feedback particle filter is to obtain gain function which can be made
efficient by using various approximation approaches.

\subsection{Nonlinear example}

This nonlinear SDE is chosen to illustrate the tracking capability of the filter in highly nonlinear settings,
\begin{subequations}
\begin{align}
\ud X_t &=  X_t (1-X_t^2) \ud t + \sigma_B \ud B_t,\label{eqn:Signal_Process_example}\\
\ud Z_t &= X_t \ud t + \sigma_W \ud W_t\, .\label{eqn:Obs_Process_example}
\end{align}
\end{subequations}
When $\sigma_B=0$, the
ODE~\eqref{eqn:Signal_Process_example} has two stable equilibria
at $\pm 1$.  With $\sigma_B>0$, the state of the SDE
``transitions'' between these  two ``equilibria''.
\spm{note that this is the quasi-steady state theory developed in my paper with Huisinga, and we use the same example.
\\
Why aren't we giving the dist. of $X_0$?}

				\spm{STOPPED here in exhaustion!}

\Fig{fig:fig_2states_comp} depicts the simulation
results obtained using the nonlinear feedback particle
filter~\eqref{eqn:particle_filter_nonlin_intro}, with 
  $\sigma_B=0.4$, $\sigma_W = 0.2$.    The
implementation is based on an algorithm described in Sec.~IV of~\cite{YangMehtaMeyn_cdc11}, and the details are omitted here on account of space. We initialize the simulation with
two Gaussian clusters.  After a brief period of transients,
these clusters merge into a single cluster, which adequately
tracks the true state including the transition events.

\begin{figure}
\begin{center}
\Ebox{.5}{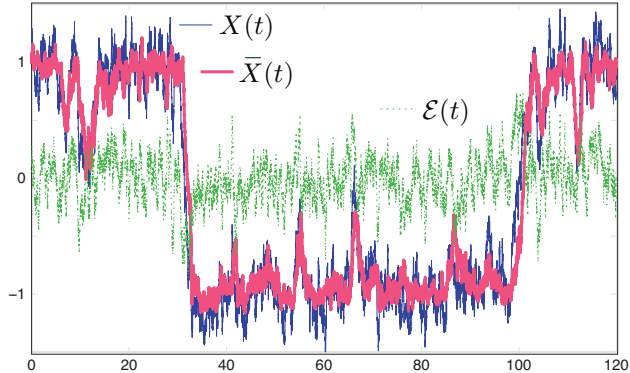}
\end{center}
\caption{Comparison of the true state $X(t)$ and the conditional mean $\bar{X}(t)$ by using feedback particle filter.  The error $\clE(t) =X(t)-\bar{X}(t)$ remains small even during a transition of the state. }
						\vspace{-.25cm}
\label{fig:fig_2states_comp}
\end{figure}

\subsection{Application: nonlinear oscillators}

\label{sub2}

We consider the filtering problem for a nonlinear oscillator:
\begin{align}
\ud \theta_t &= \omega \ud t + \sigma_B \ud B_t \quad \text{mod} \;2\pi,\label{eqn:nl_eqn1}\\
\ud Z_t &= h(\theta_t)\ud t + \sigma_W \ud W_t, \label{eqn:nl_eqn2}
\end{align}
where $\omega$ is the frequency, $h(\theta) =
\half [1+\cos(\theta)]$, and $\{B_t\}$ and $\{W_t\}$ are
mutually independent standard Wiener process. For numerical
simulations, we pick $\omega = 1$ and the standard deviation
parameters $\sigma_B = 0.5$ and $\sigma_W = 0.4$.
\spm{I think  this last line would seem odd to many readers. There isn't enough motivation.  Or, maybe it is simply misplaced -- the explanation could come at the very top of the subsection}
 We consider
oscillator models because of their significance to applications
including neuroscience; cf.,~\cite{YinMehtaMeynShanbhag12TAC}.

The feedback particle filter is given by:
\begin{align}
    \ud \theta^i_t = \omega \ud t  + \sigma_B \ud B^i_t +
    \v(\theta^i_t,t) \circ [\ud Z_t - \frac{1}{2}(h(\theta^i_t) +
    \hat{h}_t)\ud t] 
\quad \text{mod} ~2\pi, 
\label{eqn:osc_fpf}
\end{align}
$ i=1,...,N$,
where the function $\v(\theta,t)$ is obtained via the solution
of the E-L equation:
\begin{equation}
-\frac{\partial}{\partial \theta}\left(\frac{1}{p(\theta,t)}\frac{\partial }{\partial \theta}\{p(\theta,t)\v(\theta,t)\}\right) = -\frac{\sin \theta}{2\sigma_W^2}
\label{eqn:EL_nleg}.
\end{equation}

Although the equation~\eqref{eqn:EL_nleg} can be solved
numerically to obtain the optimal control function
$\v(\theta,t)$, here we investigate a solution based on
perturbation method. Suppose, at some time $t$, $p(\theta,t) =
\frac{1}{2\pi} =: p_0$, the uniform density. In this case, the
E-L equation is given by:
\begin{equation}
\partial_{\theta\theta} \v = \frac{\sin \theta}{2\sigma_W^2}.\nonumber
\end{equation}
A straightforward calculation shows that the solution in this
case is given by
\begin{equation}
\v(\theta,t) = -\frac{\sin \theta}{2\sigma^2_W} =: \v_0(\theta).
\end{equation}

To obtain the solution of the E-L equation~\eqref{eqn:EL_nleg},
we assume that the density $p(\theta,t)$ is a small harmonic
perturbation of the uniform density. In particular, we express
$p(\theta,t)$ as:
\begin{equation}
p(\theta,t) = p_0 + \epsilon \tilde{p}(\theta,t),\label{eqn:p_model}
\end{equation}
where $\epsilon$ is a small perturbation parameter. Since
$p(\theta,t)$ is a density, $\int_{0}^{2\pi}
\tilde{p}(\theta,t) \ud \theta = 0$.

We are interested in obtaining a solution of the form:
\begin{equation}
\v(\theta,t) = \v_0(\theta) + \epsilon \tilde{\v}(\theta,t).\label{eqn:ctl_nleg}
\end{equation}
On substituting the ansatz~\eqref{eqn:p_model}
and~\eqref{eqn:ctl_nleg} in~\eqref{eqn:EL_nleg}, and retaining
only $O(\epsilon)$ term, we obtain the following linearized
equation:
\begin{equation}
\partial_{\theta\theta} \tilde{\v} = -2\pi\partial_{\theta}[(\partial_{\theta}\tilde{p}) \v_0].\label{eqn:line_EL}
\end{equation}

The linearized E-L equation~\eqref{eqn:line_EL} can be solved
easily by considering a Fourier series expansion of $\epsilon
\tilde{p}(\theta,t)$:
\begin{equation}
\epsilon \tilde{p}(\theta,t) = P_c(t) \cos\theta + P_s(t) \sin\theta + \text{h.o.h},\label{eqn:fsexp_p}
\end{equation}
where ``$\text{h.o.h}$'' denotes the terms due to higher order
harmonics. The Fourier coefficients are given by,
\begin{equation}
P_c(t) = \frac{1}{\pi}\int_{0}^{2\pi} p(\theta,t)\cos \theta \ud
\theta, \; P_s(t) = \frac{1}{\pi}\int_{0}^{2\pi} p(\theta,t)\sin \theta \ud \theta.\nonumber
\end{equation}

For a harmonic perturbation, the solution of the linearized E-L
equation~\eqref{eqn:line_EL} is given by:
\begin{align}
\epsilon \tilde{\v}(\theta,t) = \frac{\pi}{4\sigma^2_W}\left( P_c(t)
  \sin 2\theta - P_s(t) \cos 2\theta\right) =: \v_1(\theta;P_c(t),P_s(t)) \label{eqn:u1}
\end{align}

For ``$\text{h.o.h}$'' terms in the Fourier series
expansion~\eqref{eqn:fsexp_p} of the density in $p(\theta,t)$,
the linearized E-L equation~\eqref{eqn:line_EL} can be solved
in a similar manner. In numerical simulation provided here, we
ignore the higher order harmonics, and use a control input as
summarized in the following proposition:
\begin{proposition}
Consider the E-L equation~\eqref{eqn:EL_nleg} where the density
$p(\theta,t)$ is assumed to be a small harmonic perturbation of
the uniform density $\frac{1}{2\pi}$, as defined
by~\eqref{eqn:p_model} and~\eqref{eqn:fsexp_p}. As $\epsilon \rightarrow
0$, the gain function is given by the following asymptotic
formula:
\spm{I added $o(\epsilon)$, since it was vacuous as stated ($\v_1=0$)!}
\begin{equation}
\v(\theta,t) = \v_0(\theta) + \v_1(\theta;P_c(t),P_s(t)) + o(\epsilon),
\label{eqn:u_mod}
\end{equation}
where $P_c(t),P_s(t)$ denote the harmonic coefficients of
density $p(\theta,t)$. For large $N$,
these are approximated by using the formulae:
\begin{equation}
P_c(t) \approx \frac{1}{\pi N}\sum_{j=1}^N \cos \theta_{j}(t), \quad P_s(t) \approx \frac{1}{\pi N}\sum_{j=1}^N \sin \theta_{j}(t).\label{eqn:pcps}
\end{equation}
\qed
\end{proposition}

We next discuss the result of numerical experiments.  The
particle filter model is given by~\eqref{eqn:osc_fpf} with gain
function $\v(\theta^i_t,t)$, obtained using
formula~\eqref{eqn:u_mod}. The number of particles $N=10,000$
and their initial condition $\{\theta^i_0\}_{i=1}^N$ was
sampled from a uniform distribution on circle $[0,2\pi]$.

\begin{figure*}
\begin{center}
\Ebox{1}{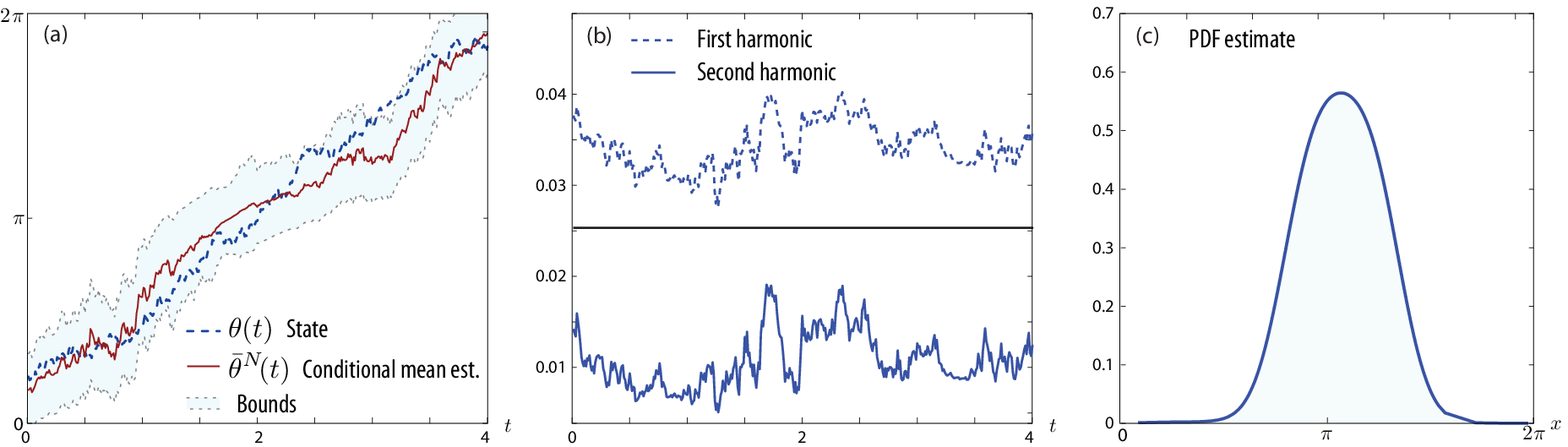}
\end{center}
\caption{Summary of the numerical experiments with the nonlinear
oscillator filter: (a) Comparison of the true state
$\{\theta_t\}$ and the conditional mean $\{\bar{\theta}^N_t\}$.
(b) The mean-squared estimate of the first and second harmonics
of the density $p(\theta,t)$ and (c) a plot of a typical
empirical distribution.}
						\vspace{-.25cm}
\label{fig_nlg}
\end{figure*}

\Fig{fig_nlg} summarizes some of the results of the
numerical simulation.  For illustration purposes, we depict
only a single cycle from a time-window after transients due to
initial condition have converged.  Part (a) of the figure
compares the sample path of the actual state $\{\theta_t\}$ (as
a dashed line) with the estimated mean
$\{\bar{\theta}^{(N)}_t\}$ (as a solid line). The shaded area
indicates $\pm$ one standard deviation bounds. Part~(b) of the
figure provides a comparison of the magnitude of the first and
the second harmonics (as dashed and solid lines, respectively)
of the density $p(\theta,t)$.  The density at any time instant
during the time-window is approximately harmonic (see also
part~(c) where the density at one typical time instant is
shown).

Note that at each time instant $t$, the estimated mean, the
bounds and the density $p(\theta,t)$ shown here are all
approximated from the ensemble $\{\theta^i_t\}_{i=1}^N$. For
the sake of illustration, we have used a Gaussian mixture
approximation to construct a smooth approximation of the
density.

\section{Conclusions}

In this paper, we introduced a new formulation of the nonlinear filter,
referred to as the {\em feedback particle filter}.  The feedback
particle filter provides for a generalization of the Kalman
filter to a general class of nonlinear non-Gaussian problems.
Feedback particle filter inherits
many of the properties that has made the Kalman filter so widely
applicable over the past five decades, including innovation error and
the feedback structure (see \Fig{fig:fig_FPF_KF}).

Feedback is important on account of the issue of {\em robustness}.  In
particular, feedback can help reduce the high
variance that is sometimes observed in the conventional particle
filter. Numerical results are presented to support this claim (see
\Fig{fig:fig_comp}).

Even more significantly, the structural aspects of the Kalman filter have been as
important as the algorithm itself in design, integration, testing and
operation of a larger system involving filtering problems
(e.g., navigation systems).  We expect feedback particle filter to
similarly provide for an integrated framework, now for nonlinear
non-Gaussian problems.  We refer the reader to our
paper~\cite{YangHuangMehta12ACC} where feedback particle filter-based
algorithms for nonlinear filtering with data association uncertainty are described.

\section{Appendix}

\subsection{Calculation of KL divergence}
\label{cal_KL}

Recall the definition of K-L divergence for densities,
\begin{equation}
\KL \left(p_n \|  \hap^*_n \right) = \int_{\Re} p_n(s) \ln \Bigl( \frac{p_n(s)}{\hap^*_n(s)} \Bigr) \ud s.\nonumber
\end{equation}
We make a co-ordinate transformation $s = x+v(x)$ and
use~\eqref{eqn:pks_2} to express the K-L divergence as:
\notes{Is this right:  $ \hap^*_n(x+u(x)|Z_{t_n})$?}
\begin{align}
\KL \left(p_n \| \hap^*_n\right) = \int_{\Re} \frac{p_n^-(x)}{|1+v'(x)|} \ln (\frac{p_n^-(x)}{|1+v'(x)|  \,  \hap^*_n(x+v(x))})|1+v'(x)| \ud x
 \nonumber
\end{align}
The expression for K-L divergence given in~\eqref{eqn:KLD}
follows on using~\eqref{eqn:pks_1}.

\subsection{Solution of the optimization problem}
\label{der_EL}

Denote:
\begin{equation}
\mathcal{L} (x,v,v') = -p_n^{-}(x)  \Bigl( \ln |1+v'| + \ln(p_n^{-}(x+v)\pyx(Y_{t_n}|x+v))  \Bigr).
\label{eqn:LAG}
\end{equation}
The optimization problem~\eqref{eqn:optimiz_problem} is a
calculus of variation problem:
\begin{equation}
\min_{v} \int \mathcal{L} (x,v,v') \ud x.\nonumber
\end{equation}
The minimizer is obtained via the analysis of first variation
given by the well-known Euler-Lagrange equation:
\begin{equation}
\frac{\partial \mathcal{L}}{\partial v} = \ddx \left(\frac{\partial \mathcal{L}}{\partial v'}\right),\nonumber
\end{equation}
Explicitly substituting the expression~\eqref{eqn:LAG} for
$\mathcal{L}$, we obtain~\eqref{eqn:EL_eqn}.

\subsection{Derivation of the Forward Equation}
\label{apdx:Derivation_FPK}

We denote the filtration $\clB_t = \sigma(X^i_0,B^i_s: s\le
t)$,  and we recall that $\clZ_t = \sigma(Z_s:s\le t)$ for
$t\ge 0$.   These two filterations are \textit{independent} by
construction.

On denoting $\tilde{a}(x,t)= a(x) + u(x,t)$, the particle
evolution \eqref{eqn:particle_model} is expressed,
\begin{equation}
X^i_t = X^i_0 + \int_0^t \tilde{a}(X_s^i,s) \ud s + \int_0^t
\v(X^i_s,s) \ud Z(s) + \sigma_B B^i_t.
\end{equation}
By assumption on Lipschitz continuity of $\tilde{a}$ and $\v$,
there exists a unique solution that is adapted to the larger
filtration $\clB_t \vee \clZ_t = \sigma(X^i_0,B^i_s, Z_s: s\le t)$. In fact, there is a functional $F_t$ such
that,
\spm{I wish we could go to the Doeblin sample path construction and not worry about weak solutions}
\begin{equation}
X_t^i = F_t(X^i_0,B^i_t,Z^t),
\label{eqn:Xti_fn_form}
\end{equation}
where $Z^t:=\{Z_s:0 \le s \le t \}$ denotes the trajectory.

The conditional distribution of $X^i_t$ given $\clZ_t =
\sigma(Z_s:s\le t)$ was introduced  in \Sec{s:believe}:
Its density is denoted $p(x,t)$, defined by any bounded and
measurable function $f\colon\Re\to\Re$ via,
\[
{\sf E}[f(X_t^i)\mid \clZ_t ] = \int_{\Re} p(x,t) f(x) \ud x =: \langle p_t,f\rangle .
\]

We begin with a result that is the key to proving \Proposition{thm:FPK}.
The proof
of \Lemma{lem:swap} is omitted on account of space.
 \notes{we need a reference for
similar results to lem:swap in the literature}

\begin{lemma}
\label{lem:swap}
Suppose that $f$ is an 
$\clB_t \vee \clZ_t $-adapted process  satisfying $ \Expect
\int_0^t |f(s)|^2 \ud s  < \infty$. Then, 
\begin{align}
{\sf E} \Bigl[\int_0^t f(s) \ud s | \clZ_t \Bigr] &= \int_0^t
{\sf E} [f(s) | \clZ_s ] \ud s,
\label{eqn:int_f}
\\
{\sf E}\Bigl[\int_0^t f(s) \ud Z_s | \clZ_t \Bigr] &= \int_0^t
{\sf E}[f(s) | \clZ_s ] \ud Z_s.\label{eqn:int_g}
\end{align}
\qed
\end{lemma}

We now provide a proof of the Proposition~\ref{thm:FPK}.  

\emph{Proof of Proposition~\ref{thm:FPK}}
Applying It\^o's formula to equation~\eqref{eqn:particle_model}
gives, for any smooth and bounded function $f$,
\[
\ud f(X_t^i) = \clL f(X_t^i) \ud t + \v(X_t^i,t) \frac{\partial
  f}{\partial x}(X_t^i) \ud Z_t+ \sigma_B \frac{\partial
  f}{\partial x}(X_t^i) \ud B_t^i,
\]
where $\clL f := (a+u)\frac{\partial f}{\partial x} +
\frac{1}{2} (\sigma_W^2 \v^2 +\sigma_B^2)\frac{\partial^2
f}{\partial x^2}$. \notes{This is just part of the heuristic
derivation of Ito's formula:   $\ud Z^2 = \sigma_B^2 \ud t$.}
Therefore,
\[
 f(X_t^i) = f(X_0^i) + \int_0^t \clL f(X_s^i) \ud s + \int_0^t \v(X_s^i,s) \frac{\partial
  f}{\partial x}(X_s^i) \ud Z_s + \sigma_B \int_0^t \frac{\partial
  f}{\partial x}(X_s^i) \ud B_s^i.
\]
Taking conditional expectations on both sides,
\begin{align*}
\langle p_t,f\rangle    = {\sf E}[f(X_0^i)\mid \clZ_t ] & +   {\sf E}\Bigl[ \int_0^t \clL f(X_s^i) \ud s |\clZ_t\Bigr]
+  {\sf E}\Bigl[ \int_0^t \v(X_s^i,s) \frac{\partial   f}{\partial x}(X_s^i) \ud Z_s \mid \clZ_t\Bigr]
  \\
  &+ \sigma_B  {\sf E}\Bigl[ \int_0^t \frac{\partial
  f}{\partial x}(X_s^i) \ud B_s^i  \mid \clZ_t\Bigr]
\end{align*}
On applying
\Lemma{lem:swap}, and the fact that $B_t^i$ is a Wiener
process, we conclude that
\[
\langle p_t,f\rangle    =
 \langle p_0,f\rangle  + \int_0^t \langle p_s,Lf\rangle  \ud s + \int_0^t \langle p_s, \v \frac{\partial
  f}{\partial x} \rangle \ud Z_s \, .
\]
The forward equation~\eqref{eqn:mod_FPK} follows using
integration by parts. \qed





\subsection{Euler-Lagrange equation for the continuous-time filter}
\label{apdx:EL_uv}

In this section we describe, formally, the continuous-time limit of
the discrete-time E-L BVP~\eqref{eqn:EL_eqn}.
In the continuous-time case, the control and the observation models
are of the form (see~\eqref{eqn:particle_model}
and~\eqref{eqn:Obs_Process}):
\begin{align*}
\ud U^i_t & = u(X^i_t, t)\ud t + \v(X^i_t, t)\ud Z_t,\\
\ud Z_t & = h(X_t) \ud t + \sigma_W \ud W_t.
\end{align*}
In discrete-time, these are approximated as
\begin{align}
\sTriangle U^i_t & =  u(X^i_t, t)\sTriangle t + \v(X^i_t, t)\sTriangle Z_t,
\label{eqn:control_ansatz}\\
\sTriangle Z_t & =  h(X_t) \sTriangle t + \sigma_W \sTriangle W_t,\nonumber
\end{align}
where $\sTriangle t$ is the small time-increment at $t$.  It follows that the conditional distribution of
 $Y_t\doteq \frac{\sTriangle Z_t}{\sTriangle t}$ given $X_t$ is the density,
\begin{equation}
\pyx(Y_t | \cdot ) =
\frac{1}{\sqrt{2\pi\sigma_W^2/\sTriangle t}}\exp\left(-\frac{(\sTriangle Z_t
    -h(\cdot)\sTriangle t)^2}{2\sigma_W^2\sTriangle t}\right).
\label{eqn:pyx_ct_dt}
\end{equation}

Substituting~\eqref{eqn:control_ansatz}-\eqref{eqn:pyx_ct_dt} in the E-L BVP~\eqref{eqn:EL_eqn} for the
continuous-discrete time case, we arrive at the formal equation:
\begin{align}
\frac{\partial}{\partial x}\left(\frac{p(x,t)}{1+u'\sTriangle t +
    \v'\sTriangle Z_t}\right) = \left.
 p(x,t)\frac{\partial}{\partial v} \Bigl( \ln p(x+v,t)
 + \ln \pyx (Y_t \mid x+v)\Bigr) \right|_{v=u\sTriangle t + \v
 \sTriangle Z_t}.\label{eqn:ctns_EL}
\end{align}

For notational ease, we
use primes to denote partial
derivatives with respect to $x$: $p$ is used to denote $p(x,t)$,
$p':= \frac{\partial p}{\partial x}(x,t)$, $p'' := \frac{\partial^2
  p}{\partial x^2}(x, t)$, $u':=\frac{\partial u}{\partial x}(x,t)$,
$\v':=\frac{\partial \v}{\partial x}(x,t)$ etc.  Note that the
time $t$ is fixed.

A sketch of calculations to obtain~\eqref{eqn:EL_v*}
and~\eqref{eqn:u_intermsof_v*} starting from~\eqref{eqn:ctns_EL}
appears in the following three steps:




\smallskip

\noindent {\bf Step 1:}
The three terms in~\eqref{eqn:ctns_EL} are simplified as:
\begin{align}
&
\frac{\partial}{\partial x} \left(\frac{p}{1+u'\sTriangle t+\v'\sTriangle Z_t}\right) = p'- f_1 \sTriangle t -(p'\v'+p\v'')\sTriangle Z_t
	\nonumber\\
&\left. p\frac{\partial }{\partial v}\ln p(x+v) \right|_{v=u\sTriangle t + \v
 \sTriangle Z_t} = p' +
f_2 \sTriangle t+ (p''\v - \frac{p'^2 \v}{p})\sTriangle Z_t
	\nonumber\\
& \left. p\frac{\partial }{\partial v}\ln \pyx (Y_t|x+ v) \right|_{v=u\sTriangle t + \v
 \sTriangle Z_t} = \frac{p}{\sigma_W^2}(h' \sTriangle Z_t -hh'\sTriangle t) + p h''\v\sTriangle t\nonumber
\end{align}
where we have used It\^{o}'s rules $(\sTriangle Z_t)^2 = \sigma_W^2 \sTriangle t$,
$\sTriangle Z_t \sTriangle t =0$ etc.,  and where
\begin{align}
&f_1 = (p'u'+pu'') - \sigma_W^2 (p'\v'^2+2p\v'\v''),\nonumber\\
&f_2 = (p''u - \frac{p'^2 u}{p}) + \sigma_W^2 \v^2 \left(\frac{1}{2} p'''-\frac{3 p'p''}{2p}+\frac{p'^3}{p^2}\right).\nonumber
\end{align}


Collecting terms in $O(\sTriangle Z_t)$ and $O(\sTriangle t)$, after some
simplification, leads to the following ODEs:
\begin{align}
{\cal E}(\v) &= \frac{1}{\sigma_W^2} h'(x) \label{eqn:EL_v_appdx}\\
{\cal E}(u) &= -\frac{1}{\sigma_W^2} h(x)h'(x) + h''(x) \v + \sigma_W^2 G(x,t)\label{eqn:EL_u}
\end{align}
where $ {\cal E}(\v) = - \frac{\partial}{\partial x}\left(
\frac{1}{p(x,t)} \frac{\partial
  }{\partial x} \{ p(x,t)\v(x,t) \}\right),
$
and $G = -2 \v'\v'' - (\v')^2 (\ln p)' + \frac{1}{2}\v^2 (\ln
p)'''$.

\smallskip

\noindent {\bf Step 2.}  Suppose $(u,\v)$ are admissible
solutions of the E-L
BVP~\eqref{eqn:EL_v_appdx}-\eqref{eqn:EL_u}. Then it is claimed
that
\begin{align}
-(p\v)' &= \frac{h-\hat{h}}{\sigma_W^2}p\label{eqn:pv} \\
-(pu)' &= -\frac{(h-\hat{h})\hat{h}}{\sigma_W^2}p - \frac{1}{2} \sigma_W^2 (p \v^2)''.\label{eqn:ns2}
\end{align}
Recall that admissible here means
\begin{equation}
\lim_{x\rightarrow\pm \infty} p(x,t) u(x,t) = 0,\quad
\lim_{x\rightarrow\pm \infty} p(x,t) \v(x,t) = 0.
\label{eqn:EL_BC_appdx}
\end{equation}

To show~\eqref{eqn:pv}, integrate~\eqref{eqn:EL_v_appdx} once to obtain
\[
-(p\v)' = \frac{1}{\sigma_W^2} h p + C p,
\]
where the constant of integration
$C=-\frac{\hat{h}}{\sigma_W^2}$ is obtained by integrating once
again between $-\infty$ to $\infty$ and using the boundary
conditions for $\v$~\eqref{eqn:EL_BC_appdx}.  This
gives~\eqref{eqn:pv}.

To show~\eqref{eqn:ns2}, we denote its right hand side as
$\mathcal{R}$ and claim
\begin{equation}
\left(\frac{\mathcal{R}}{p}\right)' = -\frac{hh'}{\sigma_W^2}+h''\v+\sigma_W^2 G.\label{eqn:ns3}
\end{equation}
The equation~\eqref{eqn:ns2} then follows by using the
ODE~\eqref{eqn:EL_u} together with the boundary conditions for
$u$~\eqref{eqn:EL_BC_appdx}.  The verification of the claim
involves a straightforward calculation, where we
use~\eqref{eqn:EL_v_appdx} to obtain expressions for $h'$ and
$\v''$. The details of this calculation are omitted on account
of space.


\smallskip

\noindent {\bf Step 3.}  The E-L equation for $\v$ is given
by~\eqref{eqn:EL_v_appdx} which is the same
as~\eqref{eqn:EL_v*}.
The proof of~\eqref{eqn:u_intermsof_v*} involves a short calculation
starting from~\eqref{eqn:ns2}, which is simplified to the
form~\eqref{eqn:u_intermsof_v*} by using~\eqref{eqn:pv}.

\begin{remark}
The derivation of Euler-Lagrange equation, as presented above,
is a heuristic on account of Step 1.  A similar heuristic also
appears in the original paper of Kushner~\cite{Kushner64SIAM}.
There, the Kushner-Stratonovich PDE~\eqref{eqn:Kushner_eqn} is
derived by considering a continuous-time limit of the Bayes
formula~\eqref{eqn:pks_1}.  The It\^{o}'s rules are used to
obtain the limit.  Rigorous justification of the calculation in
Step 1, or its replacement by an alternate argument is the
subject of future work.

The calculation in Steps 2 and 3 require additional regularity
assumptions on density $p$ and function $h$: $p$ is $C^3$ and $h$ is $C^2$.
\end{remark}


\subsection{Proof of
  Proposition~\ref{prop:existence_uniqueness_properties_EL}.}
\label{apdx:uniqueness}

Consider the ODE~\eqref{eqn:EL_v*}.  It is a linear ODE whose
unique solution is given by

{
\begin{equation}
\v(x,t) = \frac{1}{p(x,t)}\left(C_1 + C_2 \int_{-\infty}^x p(y,t) \ud y -
  \frac{1}{\sigma_W^2} \int_{-\infty}^x h(y) p(y,t) \ud y\right),
\label{eqn:closed_form_solution_of_BVP}
\end{equation}}
\noindent where the constant of integrations $C_1 = 0$ and $C_2
= \frac{\hat{h}_t}{\sigma_W^2}$ because of the boundary
conditions for $\v$.  Part~2 is an easy consequence of the
minimum principle for elliptic PDEs~\cite{Evans:98}.

\subsection{Proof of \Theorem{thm:kushner}}
\label{apdx:consistency_pf}

It is only necessary to show that with this choice of $\{
u,\v\}$,  we have $\ud p(x,t) = \ud p^*(x,t)$, for all $x$ and
$t$,  in the sense that they are defined by identical
stochastic differential equations.   Recall $\ud p^*$ is
defined according to the K-S equation~\eqref{eqn:Kushner_eqn},
and $\ud p$ according to the forward
equation~\eqref{eqn:mod_FPK}.

If $\v$ solves the E-L BVP~\eqref{eqn:EL_v*} then
using~\eqref{eqn:closed_form_solution_of_BVP},
\begin{equation}
\frac{\partial }{\partial x}(p \v) =
-\frac{1}{\sigma_W^2}(h-\hat{h})p.
\label{eqn:ppvp}
\end{equation}
On multiplying both sides of~\eqref{eqn:u_intermsof_v*} by $-p$, we
have
\begin{align*}
-up & = \frac{1}{2}(h-\hat{h})p \v - \frac{1}{2} \sigma_W^2 (p \v)
\frac{\partial \v}{\partial x}  + \hat{h} p \v \\
& = -  \frac{1}{2} \sigma_W^2  \frac{\partial (p \v)}{\partial x} \v - \frac{1}{2} \sigma_W^2 (p \v)
\frac{\partial \v}{\partial x}  + \hat{h} p \v\\
& = -  \frac{1}{2} \sigma_W^2 \frac{\partial }{\partial x}(p \v^2) + \hat{h} p \v,
\end{align*}
where we have used~\eqref{eqn:ppvp} to obtain the second equality.
Differentiating once with respect to $x$ and using~\eqref{eqn:ppvp}
once again,
\begin{equation}
 -\frac{\partial}{\partial x}(u p) +  \frac{1}{2} \sigma_W^2
 \frac{\partial^2}{\partial x^2}(p \v^2) = - \frac{\hat{h}}{\sigma_W^2}
 (h-\hat{h})p.
\label{eqn:pupp}
\end{equation}

Using~\eqref{eqn:ppvp}-\eqref{eqn:pupp} in the forward
equation~\eqref{eqn:mod_FPK}, we have
\begin{align*}
\ud p & =  \clL^\dagger p + \frac{1}{\sigma_W^2}( h-\hat{h} )(\ud Z_t -
\hat{h} \ud t)p\, .
\end{align*}
This is precisely the SDE
\eqref{eqn:Kushner_eqn}, as desired.

\subsection{Proof of \Theorem{thm_lin}}
\label{pf_thm1}

The Gaussian density is given by:
\begin{equation}
p(x,t) = \frac{1}{\sqrt{2 \pi \Sigma_t}} \exp(-\frac{(x-\mu_t)^2}{2\Sigma_t}),\label{eqn:l21}
\end{equation}

The density~\eqref{eqn:l21} is a function of the stochastic
process $\mu_t$.  Using It\^o's formula,
 \[
 \ud p(x,t) = \frac{\partial p}{\partial \mu} \ud \mu_t +
 \frac{\partial p}{\partial \Sigma} \ud \Sigma_t +
 \frac{1}{2}\frac{\partial^2 p}{\partial \mu^2}\ud \mu^2_t,
 \]
 where $ \frac{\partial p}{\partial \mu} =
 \frac{x-\mu_t}{\Sigma_t}p$, $\frac{\partial p}{\partial \Sigma}
 =
 \frac{1}{2\Sigma_t}\left(\frac{(x-\mu_t)^2}{\Sigma_t}-1\right)
 p $, and $\frac{\partial^2 p}{\partial \mu^2} =
 \frac{1}{\Sigma_t}\left(\frac{(x-\mu_t)^2}{\Sigma_t}-1\right)
 p$. Substituting these into the forward equation~\eqref{eqn:mod_FPK}, we obtain a quadratic equation $
 A x^2 + B x =0$, where
 \begin{align}
 A &= \ud \Sigma_t - \left(2\alpha\Sigma_t +\sigma_B ^2 - \frac{\gamma ^2 \Sigma^2_t}{\sigma^2_W}\right)\ud t,\nonumber\\
 B &= \ud \mu_t - \left(\alpha\mu_t\ud t + \frac{\gamma \Sigma_t}{\sigma^2_W}(\ud Z_t-\gamma \mu_t\ud t)\right).\nonumber
 \end{align}
 This leads to the model~\eqref{eqn:mod1} and~\eqref{eqn:mod2}.
 \qed

\spm{don't understand section heading - the subsection seems to be about KF and other misc. topics}
\subsection{BVP for Multivariable Feedback Particle Filter}
\label{appdx:sol_multi}

Consider the multivariable linear system,
\begin{subequations}
\begin{align}
\ud X_t &= \alpha X_t \ud t + \sigma_B \ud B_t \label{eqn:multi1}\\
\ud Z_t &= \gamma^T X_t \ud t + \sigma_W \ud W_t \label{eqn:multi2}
\end{align}
\end{subequations}
where $X_t \in \Re^d$, $Z_t \in \Re^1$, $\alpha$ is an $d
\times d$ matrix, $\gamma$ is an $d \times 1$ vector, $\{B_t\}$
is an $d-$dimensional Wiener process, $\{W_t\}$ is a scalar
Wiener process, and $\{B_t\}$,$\{W_t\}$ are assumed to be
mutually independent. We assume the initial distribution
$p^\ast (x,0)$ is Gaussian with mean vector $\mu_0$ and
variance matrix $\Sigma_0$.

The following proposition shows that the Kalman gain is a
solution of the multivariable BVP~\eqref{e:bvp_divergence_multi}, the
Kalman gain solution does not equal the solution $\v_g$
(see~\eqref{e:integral_form_of_solution_multi}), and that the solution
given by $\v_g$ is not integrable with respect to $p$:

\begin{proposition}
\label{prop_multi} Consider the d-dimensional linear
system~\eqref{eqn:multi1}-\eqref{eqn:multi2}, where $d\ge 2$.
Suppose $p(x,t)$ is assumed to be Gaussian: $p(x,t) =
\frac{1}{(2\pi)^{\frac{d}{2}} |\Sigma_t|^\frac{1}{2}} \exp
\left(-\frac{1}{2}(x-\mu_t)^T \Sigma_t^{-1} (x-\mu_t)\right)$,
where $x = (x_1,x_2,...,x_d)^T$, $\mu_t$ is the mean,
$\Sigma_t$ is the covariance matrix, and $|\Sigma_t|>0$ denotes
the determinant.
\begin{enum}
\item One solution of the
    BVP~\eqref{e:bvp_divergence_multi} is given by the
    Kalman gain:
    \begin{equation}
    \v (x,t) = \frac{1}{\sigma_W^2}\Sigma_t \gamma
    \label{eqn:multi_kalman}
    \end{equation}

\item Suppose that the Kalman gain $K$ given in
    \eqref{eqn:multi_kalman} is non-zero. For this solution to
    \eqref{e:bvp_divergence_multi}, there does not exist a
    function $\phi$ such
  that $p \v = \nabla \phi$.  
  \spm{I am still unhappy in July 2012.
  \\
  Written many months ago, let's discuss:
  HELP:  Have we explained relevance of $p \v = \nabla \phi$?  Should we say: " That is,  \eqref{eqn:pv_is_gradient} is violated"}

\item
Consider the solution $\v_g(x,t)$ of the
  BVP~\eqref{e:bvp_divergence_multi} as given
  by~\eqref{e:integral_form_of_solution_multi}.
  This gain function is unbounded:  $|\v_g(x,t)|\rightarrow\infty$ as
  $|x|\rightarrow\infty$, and moreover
  $$\int_{\Re^d} |\v_g(x,t)|p(x,t) \ud x = \infty,\quad \int_{\Re^d}
  |\v_g(x,t)|^2 p(x,t) \ud x = \infty.$$
\end{enum}
\end{proposition}

\begin{proof}
The Kalman gain solution~\eqref{eqn:multi_kalman} is     verified by direct substitution in the
    BVP~\eqref{e:bvp_divergence_multi} where the distribution $p$ is Gaussian.

 The proof of claim~2 follows by contradiction.
    Suppose a function $\phi$ exists such that $p \v = \nabla \phi$,
    then we have
    \begin{equation}
    \frac{\partial (\v p)_i}{\partial x_j} = \frac{\partial^2 \phi}{\partial x_j \partial x_i} = \frac{\partial^2 \phi}{\partial x_i \partial x_j} = \frac{\partial (\v p)_j}{\partial x_i},\quad \forall i,j\label{eqn:gradientVerify}
    \end{equation}
    where $(\v p)_i$ is the $i^{\text{th}}$ entry of vector
    $\v p$. By direct evaluation, we have
    \begin{equation}
    \frac{\partial (\v p)_i}{\partial x_j} = 2 \v_i \cdot \left(
      \Sigma_t^{-1} (x- \mu_t) \right)_j p.\nonumber
    \end{equation}
    Using~\eqref{eqn:gradientVerify}, we obtain
    \begin{equation}
    \v_i (\Sigma_t^{-1})_{jk} = \v_j (\Sigma_t^{-1})_{ik},\qquad \forall i,j,k \label{eqn:gradientCond}
    \end{equation}
    Setting $k=i$, summing over the index $i$ and
    using~\eqref{eqn:multi_kalman}, we arrive at
    \begin{equation}
    \tr(\Sigma_t^{-1}) \v  =\Sigma_t^{-1} \v,\nonumber
    \end{equation}
    where $\tr(\Sigma_t^{-1})$ denotes the trace of the matrix
    $\Sigma_t^{-1}$.  This provides a contradiction because $\v\not\equiv
    0$ and $\Sigma_t$ is a positive definite symmetric matrix with
    $|\Sigma_t|>0$.

We now establish claim~3.
For the solution $K_g$ as given by~\eqref{e:integral_form_of_solution_multi}:
\begin{equation*}
p\v_g(x,t) = \frac{1}{\sigma_W^2}\frac{1}{d\omega_d} \int \frac{y-x}{|y-x|^{d}} (h(y)-\hat{h})
p(y,t) \ud y
\end{equation*}
For this integral,  with $h(y)\equiv \gamma^T y$,  we have the
following asymptotic formula for  $|x|\sim\infty$,
\begin{equation*}
p\v_g(x,t) \sim C \frac{1}{|x|^{d}} + o(\frac{1}{|x|^{d}}),
\end{equation*}
where $C$ does not vary as a function of $|x|$ (its value depends only
upon the angular coordinates).  For example, in dimension $d=2$, $C$ is
given by
\[
C(x_1,x_2)=C(|x|\cos(\theta),|x|\sin(\theta)) =
-\frac{1}{d\omega_d}\begin{pmatrix}
\cos(2\theta) & \sin(2\theta)\\
\sin(2\theta) & \cos(2\theta)
\end{pmatrix}\frac{\Sigma_t \gamma}{\sigma_W^2},
\]
where $\frac{\Sigma_t \gamma}{\sigma_W^2} $ is the Kalman gain vector.


The result follows because $\frac{1}{p|x|^d} \rightarrow
\infty$ and $\frac{1}{|x|^{d}}$ is not integrable in $\Re^d$.  Using
the Cauchy-Schwarz inequality,
\[
\int |\v_g(x,t)|p(x,t) \ud x \le \left( \int |\v_g(x,t)|^2 p(x,t) \ud x \right)^{\frac{1}{2}},
\]
which shows that $\v_g$ is not square-integrable. 
\end{proof}

\bibliographystyle{plain}
\bibliography{strings,ACCPF,markov,CDC11,FPF_TAC}

\def\cprime{$'$}
\begin{thebibliography}{10}

\bibitem{baincrisan07}
A.~Bain and D.~Crisan.
\newblock {\em Fundamentals of Stochastic Filtering}.
\newblock Springer, Cambridge, Mass, 2010.

\bibitem{bertsi96a}
D.~P. Bertsekas and J.~N. Tsitsiklis.
\newblock {\em Neuro-Dynamic Programming}.
\newblock Atena Scientific, Cambridge, Mass, 1996.

\bibitem{budchelee07}
A.~Budhiraja, L.~Chen, and C.~Lee.
\newblock A survey of numerical methods for nonlinear filtering problems.
\newblock {\em Physica D: Nonlinear Phenomena}, 230(1-2):27 -- 36, 2007.

\bibitem{Chorin73}
A.~J. Chorin.
\newblock Numerical study of slightly viscous flow.
\newblock {\em J. Fluid Mech.}, 57:785--796, 1973.

\bibitem{Crisan_Doucet_02}
D.~Crisan and A.~Doucet.
\newblock A survey of convergence results on particle filtering methods for
  practitioners.
\newblock {\em IEEE Trans. Signal Process.}, 50(3):736--746, 2002.

\bibitem{CriXiong09}
D.~Crisan and J.~Xiong.
\newblock Approximate {McKean-Vlasov} representations for a class of {SPDEs}.
\newblock {\em Stochastics: An International Journal of Probability and
  Stochastic Processes}, pages 1--16, 2009.

\bibitem{DaumHuang10}
Fred Daum and Jim Huang.
\newblock Generalized particle flow for nonlinear filters.
\newblock In {\em Proc. SPIE}, pages 76980I--76980I--12, 2010.

\bibitem{DouFreGor01}
A.~Doucet, N.~de~Freitas, and N.~Gordon.
\newblock {\em Sequential {Monte}-{Carlo} Methods in Practice}.
\newblock Springer-Verlag, April 2001.

\bibitem{Bayesian_Brain}
K.~Doya, S.~Ishii, A.~Pouget, and R.~P.~N. Rao.
\newblock {\em Bayesian Brain}.
\newblock Comput. Neurosci. MIT Press, Cambridge, MA, 2007.

\bibitem{Evans:98}
L.~C. Evans.
\newblock {\em Partial Differential Equations}.
\newblock American Mathematical Society, 1998.

\bibitem{gorsalsmi93}
N.~J. Gordon, D.~J. Salmond, and A.~F.~M. Smith.
\newblock Novel approach to nonlinear/non-{Gaussian} {Bayesian} state
  estimation.
\newblock {\em IEE Proceedings F Radar and Signal Processing}, 140(2):107--113,
  1993.

\bibitem{Greengard:1987}
L.~Greengard and V.~Rokhlin.
\newblock A fast algorithm for particle simulations.
\newblock {\em J. Comput. Phys.}, 73:325--348, December 1987.

\bibitem{HandschinMayne69}
J.~E. Handschin and D.~Q. Mayne.
\newblock {Monte Carlo} techniques to estimate the conditional expectation in
  multi-stage nonlinear filtering.
\newblock {\em International Journal of Control}, 9(5):547--559, 1969.

\bibitem{huacaimal07}
M.~Huang, P.~E. Caines, and R.~P. Malhame.
\newblock Large-population cost-coupled {LQG} problems with nonuniform agents:
  Individual-mass behavior and decentralized $\epsilon$-{Nash} equilibria.
\newblock {\em IEEE Trans. Automat. Control}, 52(9):1560--1571, 2007.

\bibitem{kal80}
G.~Kallianpur.
\newblock {\em Stochastic filtering theory}.
\newblock Springer-Verlag, New York, 1980.

\bibitem{kun90}
H.~Kunita.
\newblock {\em {Stochastic Flows and Stochastic Differential Equations}}.
\newblock Cambridge University Press, Cambridge, 1990.

\bibitem{Kushner64SIAM}
H.~J. Kushner.
\newblock On the differential equations satisfied by conditional probability
  densities of {Markov} processes.
\newblock {\em SIAM J. on Control}, 2:106--119, 1964.

\bibitem{Leonard80}
A.~Leonard.
\newblock Vortex method for flow simulation.
\newblock {\em Journal of Computational Physics}, 37:289--335, 1980.

\bibitem{MitterNewton04}
S.~K. Mitter and N.~J. Newton.
\newblock A variational approach to nonlinear estimation.
\newblock {\em SIAM Journal on Control and Optimization}, 42(5):1813--1833,
  2003.

\bibitem{Oksendal_book}
B.~{\O}ksendal.
\newblock {\em Stochastic differential equations (6th ed.): an introduction
  with applications}.
\newblock Springer-Verlag, Inc., New York, NY, USA, 2005.

\bibitem{TiltonLizMehta12ACC}
A.~K. Tilton, E.~T. Hsiao-Wecksler, and P.~G. Mehta.
\newblock Filtering with rhythms: Application to estimation of gait cycle.
\newblock {\em In Proc. of American Control Conference}, pages 3433--3438, June
  2012.

\bibitem{Xiong}
J.~Xiong.
\newblock Particle approximations to the filtering problem in continuous time.
\newblock In D.~Crisan and B.~Rozovskii, editors, {\em The Oxford Handbook of
  Nonlinear Filtering}. Oxford University Press, 2011.

\bibitem{YangHuangMehta12ACC}
T.~Yang, G.~Huang, and P.~G. Mehta.
\newblock Joint probabilistic data association-feedback particle filter for
  multiple target tracking applications.
\newblock {\em In Proc. of American Control Conference}, pages 820--826, June
  2012.

\bibitem{YangMehtaMeyn_cdc11}
T.~Yang, P.~G. Mehta, and S.~P. Meyn.
\newblock Feedback particle filter with mean-field coupling.
\newblock {\em In Proc. of IEEE Conference on Decision and Control}, pages
  7909--7916, December 2011.

\bibitem{YangMehtaMeyn_acc11}
T.~Yang, P.~G. Mehta, and S.~P. Meyn.
\newblock A mean-field control-oriented approach to particle filtering.
\newblock {\em In Proc. of American Control Conference}, pages 2037--2043, June
  2011.

\bibitem{YinMehtaMeynShanbhag12TAC}
H.~Yin, P.~G. Mehta, S.~P. Meyn, and U.~V. Shanbhag.
\newblock Synchronization of coupled oscillator is a game.
\newblock {\em IEEE Trans. Automatic Control}, 57(4):920--935, 2012.

\end{thebibliography}
\end{document}